\documentclass[article,10pt]{article}
\usepackage{graphicx}
\usepackage{amsmath}
\usepackage{mathrsfs,amsfonts}
\usepackage{amsfonts}
\usepackage{amsmath,amssymb,amsfonts}
\usepackage{amssymb}
\usepackage{color}
\usepackage{amsmath}  
\usepackage{amsfonts} 
\usepackage{graphicx} 
\usepackage{braket}
\usepackage{enumerate}
\usepackage{indentfirst}
\usepackage{longtable}
\usepackage{setspace}
\allowdisplaybreaks[4]

\def\al{{\alpha}}\def\be{{\beta}}\def\de{{\delta}}
\def\ep{{\epsilon}}\def\ga{{\gamma}}\def\ka{{\kappa}}
\def\la{{\lambda}}\def\si{{\sigma}}
\def\ze{{\zeta}}

\def\th{{\theta}}

\def\<{\left<}\def\>{\right>}\def\({\left(}\def\){\right)}

\newfam\msbmfam\font\tenmsbm=msbm10\textfont
\msbmfam=\tenmsbm\font\sevenmsbm=msbm7

\newcommand{\blue}[1]{\textcolor{blue}{#1}}
\scriptfont\msbmfam=\sevenmsbm\def\bb#1{{\fam\msbmfam #1}}

\def\EE{\bb E}
\def\LL{\bb L}\def\NN{\bb N}\def\PP{\bb P}
\def\RR{\bb R}

\def\cF{{\cal F}}
\def\cH{{\cal H}}
\def\cM{{\cal M}}
\def\cU{{\cal U}}

\usepackage{amsmath,amsfonts,amsthm,amssymb}    
\usepackage{gensymb, mathrsfs}                  

\usepackage{accents}
\DeclareMathSymbol{\widehatsym}{\mathord}{largesymbols}{"62}

\def\*#1{\mathbf{#1}}

\newcommand{\tr}{\operatorname{tr}}

\renewcommand{\bar}{\overline}
\renewcommand{\tilde}{\widetilde}
\renewcommand{\phi}{\varphi}

\renewcommand{\qed}{\hfill \ensuremath{\Box}}

\makeatletter \renewcommand\d[1]{\ensuremath{%
  \;\mathrm{d}#1\@ifnextchar\d{\!}{}}}
\makeatother

\theoremstyle{plain}
\newtheorem{thm}{Theorem}[section]

\theoremstyle{definition}

\newtheorem{exmp}[thm]{Example}


\newcommand{\beq}{\begin{equation}}
\newcommand{\eeq}{\end{equation}}

\definecolor{c}{rgb}{0.9,0.3,0.1}
\definecolor{b}{rgb}{0.1,0.3,0.9}

\newtheorem{remark}{Remark}[section]
\newtheorem{lemma}{Lemma}[section]
\newtheorem{theorem}{Theorem}[section]
\newtheorem{definition}{Definition}[section]
\newtheorem{hypothesis}{Hypothesis}[section]
\newtheorem{Problem}{Problem}[section]
\renewcommand{\theequation}{\arabic{section}.\arabic{equation}}

\def\al{{\alpha}}\def\be{{\beta}}\def\de{{\delta}}
\def\ep{{\epsilon}}
\def\ga{{\gamma}}\def\la{{\lambda}}
\def\si{{\sigma}}\def\th{{\theta}}\def\ze{{\zeta}}

\def\<{\left<}\def\>{\right>}\def\({\left(}\def\){\right)}

\newfam\msbmfam\font\tenmsbm=msbm10\textfont
\msbmfam=\tenmsbm\font\sevenmsbm=msbm7
\scriptfont\msbmfam=\sevenmsbm\def\bb#1{{\fam\msbmfam #1}}

\def\EE{\bb E}\def\NN{\bb N}\def\PP{\bb P}
\def\RR{\bb R}

\def\cF{{\cal F}}
\def\cH{{\cal H}}\def\cM{{\cal M}}
\def \cR{{\cal R}}

\numberwithin{equation}{section}
\makeatletter
\renewenvironment{proof}[1][\proofname]{%
  \par\pushQED{\qed}%
  \normalfont \topsep6\p@\@plus6\p@\relax
  \trivlist
  \item[\hskip\labelsep\bfseries #1\@addpunct{.}]\ignorespaces
}{%
  \popQED\endtrivlist\@endpefalse
}
\makeatother

\begin{document}

\title{\large \bf Stochastic maximum principle for weighted mean-field system with jumps}
\author{Yanyan Tang\thanks{Department of Mathematics, Southern University of Science and Technology, Shenzhen, Guangdong, 518055, China ({\tt 12131233@mail.sustech.edu.cn}). } \; and 
Jie Xiong\thanks{ Department of Mathematics and SUSTech International Center for Mathematics, Southern University of Science and Technology, Shenzhen, Guangdong, 518055, China ({\tt xiongj@sustech.edu.cn}).
This author is supported by  National Key R\&D Program of China grant 2022YFA1006102, and
National Natural Science Foundation of China Grant 12471418.}
}
\date{}
\maketitle
 \bigskip

\noindent \textbf{Abstract.}
In this article, we consider a weighted mean-field control problem with jump-diffusion as its state process. The main difficulty
is from the non-Lipschitz property of the coefficients. We overcome this difficulty by an $L_{p,q}$-estimate of the solution processes 
with a suitably chosen $p$ and $q$. Convex {\blue {perturbation}} method combining with the aforementioned $L_{p,q}$-estimation method
 is utilized to derive the stochastic maximum principle for this control problem. A sufficient condition for the optimality
 is also given. Two motivating and one solvable examples are also presented.

\bigskip


\section{Introduction} \label{sec1}
\setcounter{equation}{0}
\renewcommand{\theequation}{\thesection.\arabic{equation}}

 The following stochastic system with weighted measures was studied by Kurtz and Xiong \cite{KX}: for $t\in [0,T]$, $i=1,2,\cdots$,
\begin{equation}\label{eq0121a}
\left\{\begin{array}{ccl}
dX_i(t)&=&b(X_i(t),\mu(t))dt+\si(X_i(t),\mu(t))dW_i(t),\\
dA_i(t)&=&A_i(t)\(\al(X_i(t),\mu(t))dt+\be(X_i(t),\mu(t))dW_i(t)\),\\
\mu(t)&=&\lim_{n\to\infty}\frac1n\sum^n_{i=1}A_i(t)\de_{X_i(t)},\\
X_i(0)&=&x_i.\; A_i(0)=a_i,
\end{array}
\right.\end{equation}
where $W_i$, $i=1,2,\cdots$ are independent Brownian motions.
 It was proved there that the system (\ref{eq0121a}) has a unique solution.  Further, $\mu(t)$ is characterized uniquely by a nonlinear partial differential equation. $\mu(t)$ can also be determined by the following McKean-Vlasov equation:
\begin{equation}\label{eq0121b}
\left\{\begin{array}{ccl}
dX(t)&=&b(X(t),\mu(t))dt+\si(X(t),\mu(t))dW(t),\\
dA(t)&=&A(t)\big(\al(X(t),\mu(t))dt+\be(X(t),\mu(t))dW(t)\big),\\
\mu(t)&=&\EE\(A(t)\de_{X(t)}\),\\
X(0)&=&x,\;A(0)=a,
\end{array}
\right.\end{equation}
where $W(t)$ is a Brownian motion.

In this article, we will assume that each individual can impose a suitable control to minimize certain cost function (or, equivalently, to maximize certain utility function). For example, an investor can adjust his position to maximize the profit from an investment portfolio (a detailed example will be provided in Section 6). Such control process will affect the dynamic of the  equations driving the state process $X(t)$ and the weight process $A(t)$. In this case, the state equation of the system is as follows:
\begin{equation}\label{eq0123a}
\left\{\begin{array}{ccl}
dX(t)&=&b(X(t),\mu(t),u(t))dt+\si(X(t),\mu(t),u(t))dW(t),\\
dA(t)&=&A(t)\big(\al(X(t),\mu(t),u(t))dt+\be(X(t),\mu(t),u(t))dW(t)\big),\\
\mu(t)&=&\EE\(A(t)\de_{X(t)}\),\\
X(0)&=&x,\;A(0)=a,
\end{array}
\right.\end{equation}
where $u(t)$ is the control  process to be chosen to minimize a cost functional.

{\blue{The main difficulty in proving the uniqueness of solution
to (\ref{eq0121a}) by \cite{KX} comes from the non-Lipschitz continuity of the coefficients such as $(a,\mu)\mapsto a\al(x,\mu)$.
If the coefficients are Lipschitz, then we can estimate function
\[f(t)\equiv\EE\sup_{s\le t}\(|\de X(s)|^2+|\de A(s)|^2\)\]
to obtain an inequality of the form $f(t)\le K\int^t_0 f(s)ds$, and the uniqueness follows from Gronwall's inequality, where
$\de Y=Y_1-Y_2$ for $Y=X,\ A$, and $(X_i,A_i)$, $i=1,\ 2$ are two solutions to (\ref{eq0121b}). The non-Lipschitz property mentioned above prevents us to implement the above procedure.  A stopping time
\[\tau=\inf\left\{t>0:\;\lim_{n\to\infty}\frac1n\sum^n_{i=1}A_i(t)^2>M\right\}\]
was used in \cite{KX} to overcome this difficulty. However, for the control problem we are dealing with, we will need to consider a forward backward stochastic differential equation (FBSDE) system for which stopping argument is not convenient as the terminal value of the backward component depends on that of the forward one. Namely, the equation has to be solved upto the terminal time $T$, and cannot be stopped at any earlier time.  The non-lipschitz property also causes a problem when we derive some estimates needed in the expansion of the perturbation of the cost functional.}} 

{\blue{In this article, we define
\[f(t)\equiv\EE\sup_{s\le t}|\de X(s)|^p+\(\EE\sup_{s\le t}|\de A(s)|^q\)^{p/q}\]
for suitably chosen parameters $p$ and $q$. The estimation of $f(t)$ above is called the $L_{p,q}$-estimate of the process $(\de X,\de A)$. 
When the partial dirivative of the coefficients $b$ and $\si$  with respect to the measure are bounded functions, 
we studied this control problem in \cite{T-X-2023}.  In that case, we take $p=2$ and  $q=1$. For our current setting of the partial 
derivatives  of $b$ and $\si$ with respect to the measure being of linear growth, we have to choose both parameters suitably. 
As an application of our $L_{p,q}$-method, we can solve the uniqueness problem of \cite{KX} while relaxing the boundedness condition there. 
We can also extend the results of \cite{T-X-2023} when the coefficients $b$ and $\si$, as well as the partial derivatives with respect to
 the measure, are not bounded. }}

Stochastic maximum principle (SMP) is a fundamental tool in the field of stochastic optimal control, providing necessary conditions for the optimality of controls in stochastic systems. The mean field stochastic differential equation (MFSDE) in the form of the McKean-Vlasov equation is a useful model
 for investigating the collective behavior arising from interaction among individuals, which is very popular in various fields. In the last decade, since Buchdahn et al. \cite{BDLP2009,2009+} and Carmona and Delarue \cite{{CD2013-1},{CD2013-2},{CD2015}} introduced the mean-field backward SDE (BSDE) and mean-field FBSDE, optimal control problems for mean-field systems have become a popular topic; see, for example,
\cite{ABC2019,{CD2018},{CDL2016},{CZ2016},{L-2012},WangWu2022,{Yong2013},{Z-S-X-2018}}. 
The recent seminal paper of Lasry and Lions \cite{L-L-2007} on mean-field games and their applications in economics, finance and game theory, injects a new dynamism into this research topic and opens up new avenues of applications that have attracted considerable attention from researchers. 
  
Jump diffusion models provide a natural and valuable extension of It\^o's diffusion and have attracted considerable attention due to their extensive applications, particularly in economics and finance. Several empirical studies provide evidence for the existence of jumps in stock, bond and foreign exchange markets.  In fact, many papers on {\blue {the stochastic maximum principle } }in stochastic control theory have considered the state equation  with jumps. Saksonov \cite{S-M-1985} and Tang and Li \cite{T-L-1994} proved the necessary optimality conditions for linear control problem of jump processes.  These papers impose $L^p$  boundedness conditions on the control. Abel \cite{A-C-2002} obtained an SMP for a jump model. He gives both necessary and sufficient  optimality conditions, while not imposing any $L^p$-boundness on the control. In \cite{A-C-2002} and \cite{F-N-C-2004} some applications in finance are treated. In the research by Shen and Siu \cite{S-S-2013}, they outlined the essential and comprehensive optimality conditions for a mean-field model that incorporates randomness through both Brownian motion and Poisson jump processes.  For a mean-field jump-diffusion stochastic delay system,  Zhang, et al. \cite{Z-S-X-2018} obtained a global form SMP for a switching mean-field model in the Markov regime that is driven by Brownian motion and Poisson jumps.  Using the maximum principle, Zhang  \cite{Z-2021} formulates optimality conditions for stochastic control problems associated with mean-field jump-diffusion systems characterised by both moving averages and pointwise delays. For a partially observed stochastic mean-field system with random jumps, Chen and Wu \cite{T-W-2023} studied progressive optimal control. For more details on mean-field random jump control systems and their applications, see \cite{H-M-2015}, \cite{H-M-2019} and \cite{T-M-2020}.

Based on the above consideration, we will study the control problem with the weighted mean-field state equation 
(\ref{eq0613a}) below driven by a Poisson random measure, and the cost functional given by (\ref{eq0613c1}). The goal is to derive an SMP for this problem. 
The article  is  organized as follows: Formulating the problem and presenting the main results will be the focus of Section \ref{sec2}. The existence and uniqueness of  the solution to the  system is discussed in Section \ref{sec3}.    The proof of the necessary condition of the optimality  is given in Section \ref{sec4}, while that of the sufficient condition is presented in Section \ref{sec5}. In Section 6, two motivating  and one solvable examples are introduced. We conclude this article in Section \ref{sec7}.
 Throughout the article, we will use $K$ to represent a constant whose value could be different in different places. We will omit the time parameter in complicated integrals or differentials (we omit $s-$ or $t-$ when the integral or differentials is with respect to the jump measure).

\section{Problem formulation and main results}\label{sec2}
\setcounter{equation}{0}
\renewcommand{\theequation}{\thesection.\arabic{equation}}

Let $(\Omega, \mathcal{F}, \mathbb{P})$ be a probability space with a filtration $(\mathcal{F}_{t})_{t\geq0}$ which satisfies the usual conditions. In this space, there is an $\mathbb{R}^n$-valued standard Brownian motion denoted by $W(\cdot)$ and an independent Poisson random measure
$N(dt,dz)$ with intensity  $\nu(dz)$ which is a $\sigma$-finite measure on a measurable space $(U_0,\cU_0)$. We write $\widetilde{N}(dt,dz)=N(dt,dz)-\nu(dz) dt$ for the compensated jump martingale random measure.

Consider the  McKean-Vlasov jump dynamics systems
\begin{equation}\label{eq0613a}
\left\{\begin{array}{ccl}
dX(t)&=&b(X(t),\mu^{X,A}(t),u(t))dt+\si(X(t),\mu^{X,A}(t),u(t))dW(t)\\
&&+\int_{U_0}\eta\(X(t-),u(t-),z\)\widetilde{N}(dt,dz),\\
dA(t)&=&A(t)\big(\al(X(t),\mu^{X,A}(t),u(t))dt+\be^\top(X(t),\mu^{X,A}(t),u(t))dW(t)\\
&&{\blue{+\int_{U_0}\ga(X(t-),u(t-),z)\widetilde{N}(dt,dz)}}\big),\\
X(0)&=&x,\ A(0)=a,\ t\in[0,T],
\end{array}\right.
\end{equation}
where the notation {\blue{$\top$}}  stands for the transpose of a vector or a matrix, {\blue {$u(\cdot)$ is the control process assumed to be   $\{\mathcal{F}_t\}_{t\geq 0}$-adapted and} take values in a given open convex set 
$U \subset \mathbb{R}^k$ with the condition
\begin {eqnarray*} \label{uu}
\mathbb{E}\big(\int_0^T |u(t)|^2 dt\big) < \infty.
\end {eqnarray*} 
The collection of all such control processes is denoted as $\mathcal{U}$.     Let $\cM_F(\RR^d)$ be the collection of all finite Borel measures on $\RR^d$.  Let  $\mu^{X,A}(t)\in\cM_F(\RR^d)$ be  the weighted measure given by 
\[\<\mu^{X,A}(t),f\>=\EE\(A(t)f(X(t))\), \qquad\forall f\in C_b(\RR^d).\]
The coefficients $(b,\sigma,\al, \be): \RR^d\times\cM_F(\RR^d)\times U\rightarrow\RR^{d+d\times n+1+n}$  
and $(\eta,{\blue{\ga}}):\RR^d\times U\times U_0\rightarrow\RR^{d+1}$ are given mappings.
{\blue{\begin{remark}
  The coefficients in the jump terms can also depend on the weighted measure $\mu$.  We dropped this dependence only for notational and computational  simplicity.
\end{remark}}}

 We now introduce some notation. Let $(\Omega',\mathscr{F}',\PP')$ be an independent copy of the probability space $(\Omega,\mathscr{F},\PP)$,
 $\EE'(\cdot)=\int_{\Omega'}(\cdot)d{\PP'}$ represents the expectation taken  in the product probability space $(\Omega\times\Omega,\mathscr{F}\otimes\mathscr{F}',\PP\otimes\PP')$ with respect to $\PP'$.

Let $\LL_1(\RR^d)=:\{f:|f(x)|\leq 1+|x|, \forall x\in\RR^d\} $.  {\blue{Sometimes,  we  drop 
the superscript  $X,A$ of $\mu^{X,A}_t$ for simplicity of notation.}}
\begin{definition}\label{DE2HA}
   For $\mu_1,\mu_2\in\mathcal{M}_F(\RR^d),$ the Wasserstein metric is defined by 
\[\rho(\mu_1,\mu_2)=\sup_{f\in\LL_{Lip}(\RR^d)}\{|\<\mu_1,f\>-\<\mu_2,f\>| \},\]
where $$\LL_{Lip}(\RR^d)=\{f\in\LL_1(\RR^d):\; |f(x)-f(y)|\leq|x-y|,\forall x, y\in \RR^d\},$$  
and $\<\mu,f\>$ stands for the integral of the function $f$ with respect to the measure $\mu$.
\end{definition}

Let $\LL(\RR^d)=:\{f:\;\exists K>0,\; |f(x)|\leq K(1+|x|), \forall x\in\RR^d\} $. 
\begin{definition}\label{def2.2}
Suppose $f:\mathcal{M}_F(\RR^d)\to\RR$. We say that $f\in C^1(\mathcal{M}_F(\RR^d))$ if there exists $h(\mu,\cdot)\in \LL(\RR^d)$ such that for $\ep\to 0$,
\[f(\mu+\ep\nu)-f(\mu)=\<\nu,h(\mu;\cdot)\>\ep+o(\ep).\]
We denote $h(\mu;x)=f_\mu(\mu;x)$. 
\end{definition}

Next, we make the following assumptions. We denote $\|\mu\|\equiv \sup_{f\in\LL_1(\RR^d)}\int_{\RR^d} f\mu(dx)$.
\begin{hypothesis}\label{hyp}
The coefficients  $b,\sigma,\al, \be,\eta$ are measurable in all variables. Furthermore 
 \begin{enumerate}[(1)]
\item $b,\si\in C^1(\RR^d\times\cM_F(\RR^d)\times U)$ with bounded partial derivatives in $x,u$. Furthermore, there exists a constant $K$ such that for all  $ x,\ x',\ y\in\RR^d,\ \mu,\ \mu_1,\ \mu_2\in \cM_F(\RR^d),\ u\in  U$ and  $\phi=b,\ \si$,
\begin{equation}\label{condition1}
\left\{\begin{array}{ccl}
&&|\phi(x,\mu_1,u)-\phi(y,\mu_2,u)|\leq K(|x-y|+\rho(\mu_1,\mu_2)), \\
&&|\phi_\mu(x,\mu,u; x')|\leq K(1+|x'|){\blue{,}}\\
&&|\phi(0,\mu,u)|\leq K\(1+\|\mu\|\).\\
\end{array}\right.\end{equation}
\item\label{2}For any $z\in U_0$, $\eta(x, y,z)$ is differentiable with respect to $(x, u)$. And for all $ x, y\in\RR^d, u\in  U, \  z\in U_0\  and \   p\geq2$, we have
\begin{equation}\label{condition2}
\left\{\begin{array}{ccl} 
&&\int_{U_0}|\eta(x, u,z)-\eta(y, u,z)|^p\nu(dz)\leq K|x-y|^p, \\
&&\int_{U_0}|\eta(0,  u,z)|^p\nu(dz)\leq K,\\
&&\int_{U_0}|\eta_h(x,u,z)|^p\nu(dz)\leq K,   \ \ for\  h:=x,u.
\end{array}\right.\end{equation}
\item\label{3} $\al,\be\in C^1(\RR^d\times \cM_F(\RR^d)\times U)$ are bounded and  with bounded partial derivatives.

\item\label{4} For any $z\in U_0$, $\ga(x, y,z)$ is differentiable with respect to $(x, u)$.  And   {\blue { the   processes  $1+\gamma $   are a.s. a.e. uniformly positive, i.e. there exist a constant $k$ such that $1+\gamma\geq k>0 $. Furthermore,  there exists constants $K$   such that for all $ x, y\in\RR^d, u\in  U, \  z\in U_0$  and $ p\geq1$,
\begin{equation}\label{condition3}
\left\{\begin{array}{ccl} 
&&\int_{U_0}|\ga(x,  u,z)-\ga(y, u,z)|^p\nu(dz)\leq K|x-y|^p, \\
&&  \int_{U_0}|\ga(x, u,z)|^p\nu(dz)\leq K,\\
&&\int_{U_0}|\ga_h(x, u,z)|^p\nu(dz)\leq K,   \ \ for\  h:=x,u.
\end{array}\right.\end{equation}}}
\item $\phi_\mu(x,\mu,u; x')$ is differentiable in $x'$ with bounded derivative for $\phi=b,\sigma,\al,\be$. We denote the partial derivative of  $\phi_\mu(x,\mu,u;x')$ with respect to $x'$ by $\phi_{\mu,1}(x,\mu,u;x')$.  
\end{enumerate}
\end{hypothesis}
 Under  Hypothesis \ref{hyp}  the system (\ref{eq0613a}) admits a unique solution for any $u(\cdot)\in\cU$, which will  be proved  in Section 3.
 
We will take the cost function of the form 
\begin{equation}\label{eq0613c1}
J(u)=\EE\(\int^T_0f\big(X(t),A(t),\mu^{X,A}(t),u(t)\big)dt+\Phi\big(X(T),A(T)\big)\),
\end{equation}
where $f:\RR^d\times\RR\times\mathcal{M}_F(\RR^d)\times U\rightarrow\RR$, and $\Phi:\RR^d\times\RR\rightarrow\RR.$ 
\begin{hypothesis}\label{hyp1}

\begin{enumerate}[(a)]
\item\label{a} $f(x,a,\mu,u) $ is differentiable with respect to $(x,a,\mu,u)$ and the derivatives are Lipschitz continuous.
\item\label{b} $\Phi(x,a)$ is  differentiable with respect to $(x,a)$ and the derivatives are Lipschitz continuous.
\end{enumerate}
\end{hypothesis}
The stochastic optimal control problem under consideration is stated as follows.
\begin{Problem}\label{prob}  Find  a control {\blue{ $\bar{u}(\cdot)$}} belonging to the set $\mathcal{U}$  such that
\[J(\bar{u}(\cdot))=\min_{u(\cdot)\in\cU}J(u(\cdot)).\]
{\blue{ The process  $\bar{u}(\cdot) \in \cU$  is called an optimal control of Problem \ref{prob} and we denote $(\bar{X}(\cdot), \bar{A}(\cdot))$ to be the corresponding   state process, which is the solution to equation (\ref{eq0613a})  with $u(\cdot)=\bar{u}(\cdot)$.}}
  \end{Problem}

Finally, we proceed to presenting the main results of this paper. We define the Hamiltonian function $\cH:\RR^{d+1}\times \cM_F(\RR^d)\times U \times\RR^{1+n+d+d\times n}\times\cR\rightarrow\RR$ as follows:
\begin{eqnarray}\label{eq0611H}
&&\cH\big(X,A,\mu, u, p,q, {\blue{r(\cdot)}},P,Q,R(\cdot)\big)\nonumber\\
&\equiv&A\Big(p\al(X,\mu,u)+q^\top\be(X,\mu,u)+{\blue{\int_{U_0}r(z)\ga(X, u, z)\nu(dz)}}\Big)+P^\top b(X,\mu,u)\nonumber\\
&&+\tr\big[Q^\top \si(X,\mu,u)\big]+f(X,A,\mu,u)+\int_{U_0}R(z)^\top \eta(X, u,z)\nu(dz).
\end{eqnarray}
Here, $\mathcal{R}=L^p(U_0,\cU_0,\nu;\RR^d)$. Under Hypotheses \ref{hyp} and \ref{hyp1}, it is clear that the Hamiltonian $\mathcal{H}$ exhibits continuous differentiability with respect to $(x,a,\mu,u)$.

The adjoint equations  for  the unknown adapted processes $p(t),q(t),r(t,z),P(t),Q(t)$ and $R(t,z)$ are the BSDEs 
\begin{equation}\label{eq0613adjointa3}
\left \{
  \begin{aligned}
   dp(t)=&-\big\{\cH_a(t)+\mathbb{E}'\cH_{\mu}(t';\bar{X}(t))\big\}dt+q(t) dW(t)&\\
		    &+\int_{U_0}r(t-,z)\widetilde{N}(dt,dz),&\\
   p(T)=&\Phi_a(\bar{X}(T),\bar{A}(T)),&
      \end{aligned}
  \right.
\end{equation}
and 
\begin{equation}\label{eq0613adjointa4}
\left \{
  \begin{aligned}
   dP(t)=&-\big\{\cH_x(t)+\bar{A}(t)\mathbb{E}'\cH_{\mu,1}(t';\bar{X}(t))\big\}dt+Q(t) dW(t)&\\
&+\int_{U_0}R(t-,z)\widetilde{N}(dt,dz),&\\
   P(T)=&\Phi_x(\bar{X}(T),\bar{A}(T)),&
      \end{aligned}
  \right.
\end{equation}
where we simply write
\begin{equation}\label{hmdhx1}
\cH(t):=\cH\big(\bar{X}(t),\bar{A}(t),\bar{\mu}(t), \bar{u}(t), p(t),q(t),{\blue{r(t,\cdot)}}, P(t),Q(t),R(t,\cdot)\big),
\end{equation}
\begin{equation*}
\cH_{\mu}(t'; \bar{X}):=\cH_{\mu}\big(\bar{X}'(t),\bar{A}'(t),\bar{\mu}'(t), \bar{u}'(t), p'(t),q'(t),{\blue{r'(t,\cdot)}}, P'(t),Q'(t),R'(t,\cdot);\bar{X}(t)\big),
\end{equation*}
whenever no confusion arises. And $\phi_x$ is the gradient of $\phi$ with respect to the variable $x$, the same convention also apply to variables $a$ and $u$.

The following result has been proved by Sun and Yong (\cite{S-Y-2014}, Proposition 2.1) for diffusion case without jump. The proof for the current
setup is similar and we leave to the reader.
\begin{theorem}\label{Existence}
Suppose  Hypotheses \ref{hyp} and  \ref{hyp1}  hold. Then for given $\big(X(\cdot), A(\cdot),  u(\cdot)\big)$, there exist  triples $\big(p(\cdot), q(\cdot), r(\cdot,\cdot)\big)$ and $\big (P(\cdot), Q(\cdot), R(\cdot, \cdot)\big)$  satisfying  (\ref{eq0613adjointa3}) and  (\ref{eq0613adjointa4}), respectively. 
\end{theorem}

The following is the first  main result of this paper, which presents a stochastic maximum principle for the weighted mean field control problem with jumps.

\begin{theorem}\label{smp}
(Stochastic Maximum Principle) Suppose Hypotheses \ref{hyp} and \ref{hyp1} are satisfied. If $\bar{u}(t)$ is an optimal solution to Problem \ref{prob}, then there exist a pair of processes, $\big(p(t),q(t),r(t,\cdot)\big)$ and $\big(P(t),Q(t),R(t,\cdot)\big)$, which satisfy equations (\ref{eq0613adjointa3}) and (\ref{eq0613adjointa4}) such that  
\begin{equation}\label{SMP}
\cH_u(\bar{X}(t),\bar{A}(t),\mu^{\bar{X}(t),\bar{A}(t)}(t),\bar{u}(t),p(t),q(t),{\blue{r(t,\cdot)}},P(t),Q(t),R(t,\cdot))=0,
\end{equation}
for a.e. $t\in[0,T]$, $\PP$-a.s.
\end{theorem}
{\blue{\begin{remark}
If $U$ is a convex set without assuming the openness, then (\ref{SMP}) need to be modified when $\bar{u}(t)\in\partial U$ as follows
\[\vec{n}\cdot\cH_u(\bar{X}(t),\bar{A}(t),\mu^{\bar{X}(t),\bar{A}(t)}(t),\bar{u}(t),p(t),q(t),r(t,\cdot),P(t),Q(t),R(t,\cdot))\ge 0,\]
for any vector $\vec{n}$ starting at $\bar{u}(t)$ pointing toward the inner of $U$.
\end{remark}}}

Next, we present a sufficient condition for the optimality of the weighted mean-field control problem with jumps. Besides 
 Hypotheses \ref{hyp} and \ref{hyp1}, we need the following two assumptions.
 \begin{enumerate}
 \item[(S1)]The Hamiltonian $\cH$ is convex with respect to the variables  $(x,\mu,u).$
 \item[(S2)] The terminal cost function $\Phi$ is convex with respect to the variables $(x,a).$
\end{enumerate}

\begin{theorem}\label{Sufficient}
(Sufficience) Suppose  that Hypotheses \ref{hyp}, \ref{hyp1}  and (S1)-(S2) hold. Let \linebreak[4] $\big(p(\cdot), q(\cdot), r(\cdot)\big)$ and $\big (P(\cdot), Q(\cdot), R(\cdot, \cdot)\big)$ be the unique solutions to   (\ref{eq0613adjointa3}) and  (\ref{eq0613adjointa4}), respectively. If $\bar{u}(\cdot)$   is an admissible control satisfying (\ref{SMP}). Then $\bar{u}(\cdot)$ is an optimal control process to Problem \ref{prob}. 
\end{theorem}

\section{Existence and uniqueness of the solution for the  system} \label{sec3}
\setcounter{equation}{0}
\renewcommand{\theequation}{\thesection.\arabic{equation}}

In this section, we fix $u(\cdot)\in\cU$ and consider the existence and uniqueness of the system (\ref{eq0613a}). Before presenting  the theorem, we state the following lemma.

\begin{lemma}\label{lem1}
Under (\ref{3}) and  (\ref{4}) of Hypothesis \ref{hyp},  we have 
\[{\blue{\EE\big(\sup_{t\leq T }A(t)^p\big)\leq K}},\]
for all $ p\geq1$.
\end{lemma}
\begin{proof}
   Let $\tilde{A}(t)=\ln A(t)$. Applying It\^o's  formula for $ \tilde{A}(t)$, we have 
\begin{eqnarray*}
    \tilde{A}(t)&=&\ln a_0+\int_0^t\big(\al(X, \mu, u)-\frac{1}{2}\be^\top \be (X,\mu, u)\big)ds+\int_0^t \be^\top (X,\mu,u)dW(s)\\
    &&+\int_0^t\int_{U_0}\big(\ln(1+\gamma(X,  u, z))-\ga(X,  u, z)\big)\nu(dz)ds\\
    &&+\int_0^t\int_{U_0}\ln(1+\gamma(X,  u, z))\tilde{N}(ds,dz).
\end{eqnarray*}
Hence, 
\begin{eqnarray*}
A(t)&=&a_0\exp\Big\{\int_0^t\big(\al(X, \mu, u)-\frac{1}{2}\be^\top \be (X,\mu, u)\big)ds+\int_0^t\be^\top (X, \mu, u)dW(s)\\
&& \ \ \ \ \ \ \ \ \ +\int_0^t\int_{U_0}\big(\ln(1+\gamma(X, u, z))-\ga(X,  u, z)\big)\nu(dz)ds\\
&&\ \ \ \ \ \ \ \ \ +\int_0^t\int_{U_0} \ln(1+\gamma(X,  u, z))\tilde{N}(ds,dz)
\Big\}. 
\end{eqnarray*}
 Let $\tilde{\al}=\al-\frac{1}{2}\be^\top \be $, then 
\begin{eqnarray*}
  A^p (t)&=&a_0^p\exp\bigg\{\int_0^t\(p\tilde{\al} (X,\mu, u)\)ds+\int_0^t p\be^\top (X, \mu, u)dW(s)\\
&& \ \ \ \  \ \ \ \ \ \ +\int_0^t\int_{U_0}\big(p\ln(1+\gamma(X,  u, z))-p\ga(X, u, z)\big)\nu(dz)ds\\
&&\ \ \ \ \ \ \ \ \ \   +\int_0^t\int_{U_0} p\ln(1+\gamma(X,  u, z))\tilde{N}(ds,dz)
\bigg\}\\
&=&a_0^p\exp\bigg\{\int_0^t\int_{U_0}\big(((1+\gamma)^p-p\gamma)(X, u,z)-1\big)\nu(dz)ds\\
&&\ \ \ \ \ \ \ \ \  +\int_0^t\big((p\tilde{\al}+\frac{p^2}{2}\be^\top \be\big) (X,\mu, u))ds +\int_0^tp\be^\top(X, \mu, u)dW(s)\\
&&\ \ \ \ \ \ \ \ \ -\frac{p^2}{2}\int_o^t\be^\top\be(X, \mu, u)ds+\int_0^t\int_{U_0}\ln(1+\ga)^p(X,  u, z)\tilde{N}(ds, dz)\\
&&\ \ \ \ \ \ \ \ \ +\int_0^t\int_{U_0}\big(\ln(1+\gamma)^p(X,  u, z)+1-(1+\ga)^p(X,  u, z)\big)\nu(dz)ds\bigg\}.
 \end{eqnarray*}  
 Since
 \begin{eqnarray*}
  M(t)&:=&\exp\bigg\{ \int_0^tp\be^\top(X, \mu, u)dW(s)-\frac{p^2}{2}\int_o^t\be^\top\be(X, \mu, u) ds \\
  &&\qquad +\int_0^t\int_{U_0}\ln(1+\ga)^p(X,  u, z)\tilde{N}(ds, dz)\\
&&\qquad +\int_0^t\int_{U_0}\big(\ln(1+\gamma)^p(X,  u, z)+1-(1+\ga)^p(X,  u, z)\big)\nu(dz)ds\bigg\} 
 \end{eqnarray*}
 is an exponential martingale, Girsanov theorem implies that 
\begin{eqnarray*}
   \EE\big(A(t)^p\big)&\leq& a^p\tilde{\EE}\exp\bigg\{ \int_0^T\big(p\tilde{\al}+\frac{p^2}{2}\be^\top \be\big) (X,\mu, u)ds\\
  &&\ \ \ \ \ \ \ \ \ \ +\int_0^T\int_{U_0}\big(((1+\gamma)^p-p\gamma)(X,u,z)-1\big)\nu(dz)ds\bigg\}\\
  &\leq &K,
\end{eqnarray*}
  where $\tilde{\EE}$ is the expectation with respect to an equivalent probability measure. 

The boundedness of $\la$, $\be$ and $\int_{U_0}|\gamma|^p\nu(dz)$ for $p\geq 2$ imply that 
\begin{eqnarray*}
  \EE\big(\sup_{t\leq T}A(t)^p\big)&=&  \EE\sup_{t\in[0,T]}\Big|a+\int_0^t \big(A\al(X, \mu, u)ds+A\be^\top(X, \mu, u)dW(s)\big)\\
  &&\ \ \ \ \ \ \ \ +\int_0^t\int_{U_0}A\ga(X, u, z)\tilde{N}(ds,dz)\Big|^p\\
  &\leq&  K+ \EE\Big(\int_0^T\big(|A\al(X, \mu, u)|^p+|A\be(X, \mu, u)|^p\big)dt\\ 
  &&\ \ \ \ \ \ \ \  +\int_0^T\int_{U_0}|A\ga(X,  u, z)|^2\nu(dz)^{p/2}dt+\int_0^T\int_{U_0}|A\gamma(X,  u,z)|^p\nu(dz)dt\Big)\\
  &\leq& K+K\EE\int_0^TA^pdt.
\end{eqnarray*}
Gronwall's inequality then yields the desired result.
\end{proof}

\begin{theorem}
Under  Hypothesis \ref{hyp}, the system (\ref{eq0613a}) has an unique  solution.
\end{theorem}
\begin{proof}
For the existence, we take a Picard sequence with
\begin{eqnarray*}
X^{n+1}(s)&=&x+\int^s_0b(X^n,\mu^{X^n,A^n},u)dr+\int^s_0\si(X^n,\mu^{X^n,A^n},u)dW(r)\\
&&\ \ +\int_0^s \int_{U_0}\eta(X^n,  u,z)\widetilde{N}(dr,dz),
\end{eqnarray*}
and
\begin{eqnarray*}
\tilde{A}^{n+1}(s)&=&\ln a+\int^s_0\tilde{\al}(X^n,\mu^{X^n,A^n},u)dr+\int^s_0\be^\top(X^n,\mu^{X^n,A^n},u) dW(r)\\
&&{\blue{\ \ \ \ \ +\int_0^s\int_{U_0} \big(\ln(1+\gamma(X^n,u,z))-\gamma(X^n,u, z)\big)\nu(dz)dr}}\\
&&{\blue{\ \ \ \ \ +\int_0^s\int_{U_0}\ln((1+\gamma(X^n,u, z)))\tilde{N}(dr,dz)}},
\end{eqnarray*}
where $\tilde{A}(s)=\ln A(s)$ and  $\tilde{\al}=\al-\frac{1}{2}\be^\top \be.$ Then,
\begin {eqnarray} \label{*}
&&\EE\sup_{r\le s}|X^{n+1}(r)-X^n(r)|^2\nonumber\\
&\le&K\Big(\EE\int^s_0\big|b(X^n,\mu^{X^n,A^n},u)-b(X^{n-1},\mu^{X^{n-1},A^{n-1}},u)\big|^2dr
\nonumber\\
&&\ \ \ +\EE\int^s_0\big|\si(X^n,\mu^{X^n,A^n},u)-\si(X^{n-1},\mu^{X^{n-1},A^{n-1}},u)\big|^2dr\nonumber\\
&&\ \ \ +\EE\int^s_0\int_{U_0}\big|\eta(X^n, u,z)-\eta(X^{n-1},  u,z)\big|^2\nu(dz)dr\Big)\nonumber\\
&\le&K\int^s_0\(\EE|X^n-X^{n-1}|^2+\rho(\mu^{X^n,A^n},\mu^{X^{n-1},A^{n-1}})^2\)dr.
\end{eqnarray}
For any $\delta\in(0,1)$, Definition \ref{DE2HA}  and Lemma \ref{lem1}   imply
\begin{eqnarray*}
&&\rho(\mu^{X^n,A^n}(r),\mu^{X^{n-1},A^{n-1}}(r))\\
&=&\sup_{f\in\LL_{Lip}(\RR^d)}\left|\<\mu^{X^n,A^n}(r),f\>-\<\mu^{X^{n-1},A^{n-1}}(r),f\>\right|\\
&=&\sup_{f\in\LL_{Lip}(\RR^d)}\left|\EE\Big(A^n(r)f(X^n(r))-A^{n-1}(r)f(X^{n-1}(r))\Big)\right|\\
&\le&K\EE\big(|A^n(r)-A^{n-1}(r)|(1+|X^n(r)|)\big)+K\EE\(A^{n-1}(r)|X^n(r)-X^{n-1}(r)|\)\\
&\le&K\(\EE|A^n(r)-A^{n-1}(r)|^{1+\delta}\)^{\frac{1}{1+\delta}}\big(\EE|1+X^n(r)|^{1+\delta^{-1}}\big)^{\frac{\delta}{1+\delta}}\\
&&+K\big(\EE|X^n(r)-X^{n-1}(r)|^2\big)^{1/2},
\end{eqnarray*}
for all $r\in[0,s]$. By   Kunita's inequality, for $q'=1+\de^{-1}>2$, we have
\begin{eqnarray}\label{eqqq}
&&\EE\sup_{r\le s}|X^{n}(r)|^{q'}\nonumber\\
&\le& K |x|^{q'}+K\EE\int^s_0\Big(\big|b(0,\mu^{X^{n-1},A^{n-1}},u)\big|^{q'}dr\nonumber\\
&&+K\EE\int^s_0\big|b(X^{n-1},\mu^{X^{n-1},A^{n-1}},u)-b(0,\mu^{X^{n-1},A^{n-1}},u)\big|^{q'}dr\nonumber\\
&&+K\EE\int^s_0\big|\si(0,\mu^{X^{n-1},A^{n-1}}(r),u(r))\big|^{q'}dr\nonumber\\
&&+K\EE\int^s_0\big|\si(X^{n-1},\mu^{X^{n-1},A^{n-1}},u)-\si(0,\mu^{X^{n-1},A^{n-1}},u)\big|^{q'}dr\nonumber\\
&&+K\EE\Big(\int^s_0\int_{U_0}|\eta(0, u,z)|^2\nu(dz)^{{q'}/2}dr +\int^s_0\int_{U_0}|\eta(0, u, z)|^{q'}\nu(dz)dr\Big)\nonumber\\
&&+K\EE\Big(\int^s_0\int_{U_0}|\eta(X^{n-1},  u,z)-\eta(0, u,z)|^2\nu(dz)^{{q'}/2}dr\nonumber\\
&&\ \ \ \  \ \ \ \ +\int^s_0\int_{U_0}|\eta(X^{n-1},  u,z)-\eta(0,  u,z)|^{q'}\nu(dz)dr\Big)\nonumber\\
&\le&K+K\int_0^s\Big(1+\big(\sup_{f\in\LL_1(\RR^d)}\(\EE(f(X^{n-1})A^{n-1})\big)\)^{q'}+\EE|X^{n-1}|^{q'}\Big)dr\nonumber\\
&\le&K+K\int_0^s\EE|X^{n-1}|^{q'}dr,
\end{eqnarray}
 for any $n\in\NN^+$ and $r\in[0,s]$, where the last inequality follows from
 \begin{eqnarray*}
 \Big(\EE\big|(1+X^{n-1}(r))A^{n-1}(r)\big|\Big)^{q'}&\le&\EE\big(1+|X^{n-1}(r)|\big)^{q'}\big(\EE A^{n-1}(r)^{{q'}/({q'}-1)}\big)^{{q'}-1}\\
 &\le& K+ K\EE|X^{n-1}(r)|^{q'}.\end{eqnarray*}

By induction, we can prove that
\[\EE\sup_{r\le s}|X^{n}(r)|^{q'}\leq K e^{Ks}.\]
Continuing with (\ref{*}), we get 
\[\EE\sup_{r\le s}|X^{n+1}(r)-X^n(r)|^2\le K\int^s_0\Big(\EE|X^n-X^{n-1}|^2+\big(\EE|A^n-A^{n-1}|^{q}\big)^{\frac{2}{q}}\Big)dr,\]
 where $q=1+\delta$ with $\frac{1}{{q'}}+\frac{1}{q}=1$.  {\blue{Since $1+\ga\geq k>0$, we have $(1+\gamma)^{-1}\leq k^{-1}:=K$, which together with the mean value theorem yields}} 
 {\blue{ \begin{eqnarray*}
 &&\int_{U_0} |\ln(1+\gamma(X^{n}, u,z))-\ln(1+\gamma(X^{n-1}, u,z))|^p\nu(dz)\\
   & \leq& \int_{U_0}K^p|\gamma(X^{n}, u,z)-\gamma(X^{n-1}, u,z)|^p\nu(dz).
 \end{eqnarray*}}}

 Similarly, we get
\[\EE\sup_{r\le s}|\tilde{A}^{n+1}(r)-\tilde{A}^n(r)|^2\le K\int^s_0\(\EE|X^n-X^{n-1}|^2+\(\EE|A^n-A^{n-1}|^{q}\)^{\frac{2}{q}}\)dr.\]
Then,
\begin{eqnarray*}
&&\big(\EE\sup_{r\le s}|A^{n+1}(r)-A^n(r)|^{q}\big)^{\frac{2}{q}}\\
&\le&\(\EE\sup_{r\le s}(A^{n+1}(r)+A^n(r))^{q}|\tilde{A}^{n+1}(r)-\tilde{A}^n(r)|^{q}\)^{\frac{2}{q}}\\
&\le&\(\EE\sup_{r\le s}(A^{n+1}(r)+A^n(r))^{\frac{2q}{2-q}}\)^{\frac{2-q}{q}}
\EE\sup_{r\le s}|\tilde{A}^{n+1}(r)-\tilde{A}^n(r)|^2\\
&\le&K\EE\sup_{r\le s}|\tilde{A}^{n+1}(r)-\tilde{A}^n(r)|^2\\
&\le&K\int^s_0\(\EE|X^n-X^{n-1}|^2+\(\EE|A^n-A^{n-1}|^{q}\)^{\frac{2}{q}}\)dr.\end{eqnarray*}
Let
\[f^n(s)=\EE\sup_{r\le s}|X^{n+1}(r)-X^n(r)|^2+\(\EE\sup_{r\le s}|A^{n+1}(r)-A^n(r)|^{q}\)^{\frac{2}{q}}.\]
Then,
\begin{equation}\label{eq0613b}
f^n(s)\le K\int^s_0f^{n-1}(r)dr.\end{equation}
By iterating, we see that
\[f^n(T)\le K'\frac{(KT)^n}{n!},\]
which is summable. Hence, there exists process $(X,A)$ such that
\[\EE\sup_{r\le s}|X^n(r)-X(r)|^2+\(\EE\sup_{r\le s}|A^n(r)-A(r)|^{q}\)^{\frac{2}{q}}\to 0.\]
It is then easy to show that $(X,A)$ is a solution to (\ref{eq0613a}).

To prove the uniqueness, we take two solutions and denote the difference as $(Y,B)$. Let
\[g(s)=\EE\sup_{r\le s}|Y(r)|^2+\(\EE\sup_{r\le s}|B(r)|^{q}\)^{\frac{2}{q}}.\]
Similar to (\ref{eq0613b}), we get
\[g({\blue{s}})\le K\int^s_0g(r)dr,\]
and hence $g=0$. This yields the uniqueness.
\end{proof}

\section{Necessary condition}\label{sec4}
\setcounter{equation}{0}
\renewcommand{\theequation}{\thesection.\arabic{equation}}

In this section, we proceed to proving the necessary optimality condition for the weighted mean-field control problem. 
Assume $\bar{u}(\cdot)\in \mathcal{U}$ is optimal. 
Let $v(\cdot)$ be any  given process such that $\bar{u}(\cdot) + v(\cdot) \in \mathcal{U}$. Since $\cU$ is convex, the  control  
$u^\epsilon(\cdot)\equiv \bar{u}(\cdot)+\epsilon v(\cdot)$ is also in $\cU$, for  all $ \epsilon\in[0,1].$  Hence,  $J(\bar{u})\leq J(u^\epsilon)$. Let $\big(X^\epsilon(\cdot), A^\epsilon(\cdot)\big)$  be the state process with control $u^\ep(\cdot)$. In the following proof of the theorem, for  simplicity of notation, we write $\th=(\bar{X},\mu^{\bar{X}, \bar{A}},\bar{u})$, $\ka=(\bar{X},\bar{A}, \mu^{\bar{X}, \bar{A}},\bar{u})$, while $\th^\ep,\ \th^v$ and $\ka^v$ are defined similarly. 
We use superscript to denote the components of the vectors or matrices, e.g. $\si=(\si^1,\si^2,\cdots,\si^n)=(\si^{ij})_{d\times n}$.
Let
$\psi^i_x=\(\psi^i_{x^1},\psi^i_{x^2}, \cdot\cdot\cdot, \psi^i_{x^d}\)$, for $\psi=b, \si^j, \eta, \be$,\ $x=(x^1, x^2,\cdot\cdot\cdot, x^d)$ and $i=1,2, \cdot\cdot\cdot,d$.

Let $\phi$ be one of $b, \si^j, \al, \be^j$. 
The following estimate (\ref{eq0611g}) will be very useful.
For $\delta\in(0,1)$, let $q'=1+\delta^{-1}, q=1+\delta$, $p\geq2$ and $\frac{1}{p}+\frac{1}{p'}=1$.  Let $\mu^\ep=\mu^{X^\ep, A^\ep}$
and $\vartheta^\ep(s)$ between $\th^\ep(s)$ and $\th(s)$. Then,
\begin{eqnarray}\label{eq0611g}
&&\qquad \big|\<\mu^\ep(s)-\bar\mu(s),\phi_\mu(\vartheta^\ep(s);\cdot)\>\big|\\
&=&\left|\EE'\Big(\big(A^\ep(s)'-\bar{A}(s)'\big)\phi_\mu(\vartheta^\ep(s); X^\ep(s)')+\bar{A}(s)'\phi_{\mu,1}(\vartheta^\ep(s); \tilde{X}(s)')\big(X^\ep(s)'- \bar{X}(s)'\big)\Big)\right|\nonumber\\
&\le&\big(\EE'|A^\ep(s)'-\bar{A}(s)'|^{q}\big)^{\frac{1}{q}}\big(\EE'|1+X^\ep(s)'|^{q'}\big)^{\frac{1}{q'}}\nonumber\\
&&+(\EE'|\bar{A}(s)'\phi_{\mu,1}(\vartheta^\ep(s); \tilde{X}(s)')|^{p'})^{\frac{1}{p'}}(\EE|X^\ep(s)'-\bar{X}(s)'|^p\big)^{{\frac{1}{p}}}\nonumber\\
&\le&K\(\EE'|A^\ep(s)'-\bar{A}(s)'|^{q}\)^{\frac{1}{q}}+K\(\EE'|X^\ep(s)'-\bar{X}(s)'|^p\)^{\frac{1}{p}},\nonumber
\end{eqnarray}
 for any $s\in[0,T]$,  where $\tilde{X}(s)$ is between $\bar{X}(s)$ and $X^\ep(s)$, and the last inequality used  (\ref{eqqq}).
 
 We  now present a few lemmas.
 
\begin{lemma}
There exists a constant $K$ such that  for any $p\geq 2$ and $\delta\in(0,1)$,  
\begin{equation}\label{eq0610a}
\EE{\blue{\sup_{s\in[0,T]}}}|X^\ep(s)-\bar{X}(s)|^p+\big(\EE{\blue{\sup_{s\in[0,T]}}}|A^\ep(s)-\bar{A}(s)|^{q}\big)^{\frac{p}{q}}\le K\ep^p,\end{equation}
where $q=1+\delta$. 
\end{lemma}
\begin{proof}
By  Hypothesis 2.1 and (\ref{eq0611g}), for any $p\geq 2,$ $s\in[0,T]$, we  obtain
\begin{eqnarray}\label{eq061100}
&&|b(\th^\ep(s))-b(\th(s))|^p\nonumber\\
&=&\big|b_x(\vartheta^\ep(s))(X^\ep(s)-\bar{X}(s))+\<\mu^\ep(s)-\mu(s),b_\mu(\vartheta^\ep(s);\cdot)\>+\ep b_u(\vartheta^\ep(s))v(s)\big|^p\nonumber\\
&\leq& K\Big(|X^\ep(s)-\bar{X}(s)|^p+\(\EE|A^\ep(s)-\bar{A}(s)|^{q}\)^{\frac{p}{q}}+\EE|X^\ep(s)-\bar{X}(s)|^p\nonumber\\
&&\ \ \  +\ep^p|b_u(\vartheta^\ep(s))v(s)|^p\Big).
\end{eqnarray}
The same estimate holds for $|\si^j(\th^\ep(s))-\si^j(\th(s))|^p$.

We then estimate $\int_{U_0}\big|\eta(X^\ep(s),u^\ep(s),z)-\eta(\bar{X}(s),\bar{u}(s),z)\big|^p\nu(dz)$. Note that
\begin{eqnarray}\label{eq0611005}
&&\int_{U_0}\big|\eta(X^\ep(s),u^\ep(s),z)-\eta(\bar{X}(s),\bar{u}(s),z)\big|^p\nu(dz)\\
&=&\int_{U_0}\big|\eta_x(\chi^\ep(s),z)(X^\ep(s)-\bar{X}(s))+\ep \eta_u(\chi^\ep(s),z)v(s)\big|^p\nu(dz)\nonumber\\
&\leq& K\(|X^\ep(s)-\bar{X}(s)|^p\int_{U_0}|\eta_x(\chi^\ep(s),z)|^p\nu(dz)+\ep^p|v(s)|^p\int_{U_0}|\eta_u(\chi^\ep(s),z)|^p\nu(dz)\),\nonumber
\end{eqnarray}
where $\chi^\ep(s)$ is between $(X^\ep(s),u^\ep(s))$ and $(\bar{X}(s),\bar{u}(s))$. Then,   Kunita's inequality  implies 
\begin{eqnarray*}
&&\EE \blue{\sup_{s\in[0,T]}}|X^\ep(s)-\bar{X}(s)|^p\nonumber\\
&\le&K\Big(\EE\int^{\blue{T}}_0|b(\th^\ep)-b(\th)|^pds+\sum _{i=1}^n\EE\int^{\blue{T}}_0|\si^j(\th^\ep)-\si^j(\th)|^pds\nonumber\\
&&\ \ \ \  +\EE\int^{\blue{T}}_0\int_{U_0}|\eta(X^\ep,u^\ep,z)-\eta(\bar{X},\bar{u},z)|^2(\nu(dz))^{p/2}ds\nonumber\\
&&\ \ \ \ +\EE\int^{\blue{T}}_0\int_{U_0}|\eta(X^\ep,u^\ep,z)-\eta(\bar{X},\bar{u},z)|^p\nu(dz)ds\Big).\nonumber\\
\end{eqnarray*}
Combined with  (\ref{eq061100}) and (\ref{eq0611005})  we get
\begin{eqnarray}\label{eq0611h}
&&\EE\blue{\sup_{s\in[0,T]}}|X^\ep(s)-\bar{X}(s)|^p\nonumber\\
&\le&K\int^{\blue{T}}_0\(\EE|X^\ep-\bar{X}|^p+\(\EE|A^\ep-\bar{A}|^{q}\)^{\frac{p}{q}}+\ep^p\)ds\nonumber\\
&\le& K \int^{\blue{T}}_0\Big(\EE\sup_{r\in[0,s]}|X^\ep(r)-\bar{X}(r)|^p+\big(\EE\sup_{r\in[0,s]}|A^\ep(r)-\bar{A}(r)|^{q}\big)^{\frac{p}{q}}+\ep^p\Big)ds.
\end{eqnarray}

Define  $\tilde{A}^\ep(s)=\ln A^\ep(s)$ and $\tilde{A}(s)=\ln \bar{A}(s)$, we have 
\begin{eqnarray}\label{eq0611005A}
&&\EE{\blue{\sup_{s\in[0,T]}}}|\tilde{A}^\ep(s)-\tilde{A}(s)|^p\nonumber\\
&\le&K\Big(\EE\int^{\blue{T}}_0|\tilde{\al}(\th^\ep)-\tilde{\al}(\th)|^pds+\EE\int^{\blue{T}}_0|\be(\th^\ep)-\be(\th)|^pds\nonumber\\
&&\ \  {\blue{+\EE\int_0^T\int_{U_0}\big|\ln(1+\ga( X^\ep, u^\ep, z)-\ln(1+\ga(\bar{X}, \bar{u}, z)))\big|^p\nu(dz)ds}}\nonumber\\
&&\ \  {\blue{+\EE\int_0^T\int_{U_0}\big|\ga( X^\ep, u^\ep, z)-\ga(\bar{X}, \bar{u}, z)\big|^p\nu(dz)ds}}\Big)\nonumber\\
&\le& K\int^{\blue{T}}_0\(\EE|X^\ep-\bar{X}|^p+\(\EE|A^\ep-\bar{A}|^{q}\)^{\frac{p}{q}}+\ep^p\)ds,
\end{eqnarray}
where the last inequality follows from (\ref{eq061100}) with $b$ replaced by  $\tilde{\al}$, $\be$ and  {\blue{$\gamma$}}, respectively. By LaGrange's intermediate value theorem,  there exists $\ze^\ep(s)$ between $A^\ep(s)$ and $\bar{A}(s)$ such that 
\[A^\ep(s)-\bar{A}(s)=\ze^\ep(s)(\tilde{A}^\ep(s)-\tilde{A}(s)).\]
Then,
\begin{eqnarray*}
\big(\EE{\blue{\sup_{s\in[0,T]}}}|A^\ep(s)-\bar{A}(s)|^{q}\big)^{\frac{p}{q}}&=&\Big(\EE{\blue{\sup_{s\in[0,T]}}}|\ze^\ep(s)(\tilde{A}^\ep(s)-\tilde{A}(s))|^q\Big)^{\frac{p}{q}}\nonumber\\
&\leq&\big(\EE{\blue{\sup_{s\in[0,T]}}}|{\ze^\ep(s)}|^{\frac{pq}{p-q}}\big)^{\frac{p-q}{q}} \EE{\blue{\sup_{s\in[0,T]}}}|\tilde{A}^\ep(s)-\tilde{A}(s)|^p.
\end{eqnarray*}
Lemma \ref{lem1}  implies that $\EE {\blue{\sup_{s\in[0,T]}}}|{\ze^\ep(s)}|^{\frac{pq}{p-q}}\leq K$, for all $s\in[0,T]$.  Hence 
\begin{eqnarray}\label{eq0611i}
&&\big(\EE{\blue{\sup_{s\in[0,T]}}}|A^\ep(s)-\bar{A}(s)|^{q}\big)^{\frac{p}{q}}\nonumber\\
&\le& K\int^{\blue{T}}_0\Big(\EE|X^\ep-\bar{X}|^p+\big(\EE|A^\ep-\bar{A}|^{q}\big)^{\frac{p}{q}}+\ep^p\Big)ds\nonumber\\
&\blue{\le}&\blue{ K \int^{\blue{T}}_0\Big(\EE\sup_{r\in[0,s]}|X^\ep(r)-\bar{X}(r)|^p+\big(\EE\sup_{r\in[0,s]}|A^\ep(r)-\bar{A}(r)|^{q}\big)^{\frac{p}{q}}+\ep^p\Big)ds}.
\end{eqnarray}

Denote the left hand side  of (\ref{eq0610a}) by $f(\blue{T})$. Adding inequalities (\ref{eq0611h}, \ref{eq0611i}), we obtain
\[f({\blue{T}})\le K\int^{\blue{T}}_0 f({\blue{s}})ds+K\ep^p.\]
The desired conclusion follows from Gronwall's inequality.
\end{proof}

Define $Y^\ep(s)=\ep^{-1}\(X^\ep(s)-\bar{X}(s)\)$, $B^\ep(s)=\ep^{-1}(A^\ep(s)-\bar{A}(s))$  and  $\tilde{B}^\ep(s)=\ep^{-1}(\tilde{A}^\ep(s)-\tilde{A}(s))$. Then,
\begin{eqnarray}{\label{eq0611f}}
dY^\ep(s)&=&\ep^{-1}\big(b(\th^\ep)-b(\th)\big)ds+\ep^{-1}\big(\si(\th^\ep)-\si(\th)\big)dW(s)\nonumber\\
&&+\ep^{-1}\int_{U_0}\big(\eta(X^\ep, u^\ep,z)-\eta(\bar{X}, \bar{u},z)\big)\widetilde{N}(ds,dz)\nonumber\\
&=&\Big(b_x(\vartheta^\ep)Y^\ep+b_u(\vartheta^\ep)v +\EE'\big({B^\ep}' b_\mu(\vartheta^\ep; {X^\ep}')+\bar{A}'b_{\mu,1}(\vartheta^\ep; \tilde{X}'){Y^\ep}'\big)\Big)
ds\nonumber\\
&&+\sum_{j=1}^n\Big(\si_x^j(\vartheta^\ep)Y^\ep+\si^j_u(\vartheta^\ep)v+\EE'\big({B^\ep}' \si^j_\mu(\vartheta^\ep; {X^\ep}')+\bar{A}'\si^j_{\mu,1}(\vartheta^\ep; \tilde{X}'){Y^\ep}'\big)
\Big)dW^j(s)\nonumber\\
&&+\int_{U_0}\Big(\eta_x(\chi^\ep, z)Y^\ep+\eta_u(\chi^\ep ,z)v
\Big)\widetilde{N}(ds,dz),
\end{eqnarray}
and 
\begin{eqnarray}{\label{eq0611B'}}
d\tilde{B}^\ep(s)&=&\ep^{-1}\big(\tilde{\al}(\th^\ep)-\tilde{\al}
(\th)\big)dt+\ep^{-1}\big(\be(\th^\ep)-\be(\th)\big)^\top dW(s)\nonumber\\
&&{\blue{+\ep^{-1}\int_{U_0} \big(\ln(1+\gamma(X^\ep, u^\ep,z))-\ln(1+\gamma(\bar{X},\bar{u},z))\big)\nu(dz)ds}}\nonumber\\
&&{\blue{-\ep^{-1}\int_{U_0} \big(\gamma(X^\ep, u^\ep,z)-\gamma(\bar{X},\bar{u},z)\big)\big)\nu(dz)ds}}\nonumber\\
&&\blue{+\ep^{-1}\int_{U_0} \big(\ln(1+\gamma(X^\ep, u^\ep,z))-\ln(1+\gamma(\bar{X},\bar{u},z)\big)\tilde{N}(ds,dz)}\nonumber\\
&=&\Big(\tilde{\al}_x(\vartheta^\ep)Y^\ep+\tilde{\al}_u(\vartheta^\ep)v
  +\EE'\big({B^\ep}' \tilde{\al}_\mu(\vartheta^\ep; {X^\ep}')+\bar{A}'\tilde{\al}_{\mu,1}(\vartheta^\ep; \tilde{X}'){Y^\ep}'\big)
\Big)ds\nonumber\\
&&+\sum_{j=1}^n\Big(\be_x^j(\vartheta^\ep)Y^\ep+\be^j_u(\vartheta^\ep)v +\EE'\big({B^\ep}' \be^j_\mu(\vartheta^\ep; {X^\ep}')+\bar{A}'\be^j_{\mu,1}(\vartheta^\ep; \tilde{X}'){Y^\ep}'\big)
\Big)dW^j(s)\nonumber\\
&&+\int_{U_0}\Big(\frac{1}{1+\zeta^\ep}\big( \gamma_x(\chi^\ep,z)Y^\ep+\gamma_u(\chi^\ep,z)v \big)- \big(\gamma_x(\chi^\ep,z) Y^\ep+\gamma_u(\chi^\ep,z)v \big)\Big)\nu(dz)ds\nonumber\\
&&+\int_{U_0}\frac{1}{1+\zeta^\ep}\big( \gamma_x(\chi^\ep,z)Y^\ep+\gamma_u(\chi^\ep,z)v \big)\tilde{N}(ds,dz),
\end{eqnarray}
where  $\zeta^\ep(s)$ between $\gamma(X^\ep(s), u^\ep(s), z)$ and $\gamma(\bar{X}(s), \bar{u}(s), z)$,  and $\chi^\ep(s)$ is between $(X^\ep(s),u^\ep(s))$ and $(\bar{X}(s),\bar{u}(s))$, $s\in[0,T]$. 

Consider the following mean-field SDEs 
\begin{equation}{\label{B1}}
\begin{split}
  \begin{aligned}
dY(s)&=\Big(b_x(\th)Y+b_u(\th)v +\EE'\big(B'b_{\mu}(\th; \bar{X}')+\bar{A}'b_{\mu,1}(\th; \bar{X}')Y'\big)\Big)ds\\
&+\sum_{j=1}^n\Big(\si_x^j(\th)Y+\si_u^j(\th)v+\EE'\(B'\si_{\mu}^j(\th; \bar{X}')+\bar{A}'\si_{\mu,1}^j(\th; \bar{X}')Y'\)\Big)dW^j(s)\\
&+\int_{U_0}\big(\eta_x(\bar{X}, \bar{u},z)Y+\eta_u(\bar{X}, \bar{u},z)v\big)\widetilde{N}(ds,dz),\\
 \end{aligned}
  \end{split}
\end{equation}
and 
\begin{equation}{\label{B2}}
\begin{split}
  \begin{aligned}
d\tilde{B}(s)&=\Big(\tilde{\al}_x(\th)Y+\tilde{\al}_u(\th)v +\EE'\big(B'\tilde{\al}_{\mu}(\th; \bar{X}')+\bar{A}'\tilde{\al}_{\mu,1}(\th; \bar{X}')Y'\big)\Big)ds\\
&+\sum_{j=1}^n\Big(\be^j_x(\th)Y+\be^j_u(\th)v  +\EE'\big(B'\be^j_{\mu}(\th; \bar{X}')+\bar{A}'\be^j_{\mu,1}(\th; \bar{X}')Y'\big)\Big)dW^j(s)\\
&{\blue{+\int_{U_0}\frac{1}{1+\gamma(\bar{X}, \bar{u},z)}\big(\gamma_x(\bar{X}, \bar{u},z)Y+\gamma_u(\bar{X}, \bar{u},z)v \big)\nu(dz)ds}}\nonumber\\
&{\blue{ -\int_{U_0}\big(\gamma_x(\bar{X}, \bar{u},z)Y +\gamma_u(\bar{X}, \bar{u},z)v \big)\nu(dz)ds}}\nonumber\\
&{\blue{+\int_{U_0}\frac{1}{1+\ga(\bar{X}, \bar{u},z)}\big(\gamma_x(\bar{X}, \bar{u},z)Y+\gamma_u(\bar{X}, \bar{u},z)v \big)\tilde{N}(ds,dz)}},
\end{aligned}
  \end{split}
\end{equation}
with $Y_0=\tilde{B}_0=0$, $B(s)=\bar{A}(s)\tilde{B}(s)$, $s\in[0,T]$. For  $v$ being fixed and any $ p\geq2$,  under Hypothesis 2.1 the  variational equations (\ref{B1}) and (\ref{B2}) admits a unique pair of solutions $(Y(s), \tilde{B}(s))$ with $\EE\big(|Y(s)|^p+|\tilde{B}(s)|^p|\big)\leq K.$

\begin{lemma}
Suppose that  Hypothesis  \ref{hyp} is satisfied.   
For any $p\geq2$ and $\delta\in(0,1)$, we have
\begin{eqnarray}{\label{L2}}
\lim_{\ep\downarrow 0}\(\EE\sup_{s\in[0,T]}|Y^\ep(s)-Y(s)|^p+\big(\EE\sup_{s\in[0,T]}|B^\ep(s)-B(s)|^{q}\big)^{\frac{p}{q}}\)=0,
\end{eqnarray}
where $q=1+\delta.$
\end{lemma}
\begin{proof}
Applying It\^o's formula to $\bar{A}(s)\tilde{B}(s)$, we have
\begin{eqnarray*}
dB(s)&=&\bar{A}\Big(\tilde{\al}_x(\th)Y+\tilde{\al}_u(\th)v +\EE'\big(B'\tilde{\al}_{\mu}(\th; \bar{X}')+\bar{A}'\tilde{\al}_{\mu,1}(\th; \bar{X}')Y'\big)\Big)ds\\
  &&+\sum_{i=1}^n\bar{A}\Big(\be^j_x(\th)Y+\be^j_u(\th)v  +\EE'\big(B'\be^j_{\mu}(\th; \bar{X}')+\bar{A}'\be^j_{\mu,1}(\th; \bar{X}')Y'\big)\Big)dW^j(s)\\
&&+\Big(\tilde{B}\big(\bar{A}\al(\th)ds+\bar{A}\sum_{j=1}^n\be^j(\th)dW^j(s)\big)+\tilde{B}\int_{U_0}\bar{A}\ga(\bar{X}, \bar{u},z)\tilde{N}(ds,dz)\Big)\\
&&+\bar{A}\sum_{j=1}^n\be^j(\th)\Big(\be^j_x(\th)Y+\be^j_u(\th)v  +\EE'\big(B'\be^j_{\mu}(\th; \bar{X}')+\bar{A}'\be^j_{\mu,1}(\th; \bar{X}')Y'\big)\Big)ds\\
&&{\blue{+ \int_{U_0}\bar{A}\Big(\ga_x(\bar{X}, \bar{u}, z)Y+\ga_u(\bar{X}, \bar{u},z)v \Big)\tilde{N}(ds,dz).}}
\end{eqnarray*}
Since $\tilde{\al}=\al-\frac{1}{2}\be^\top \be$, we get $\tilde{\al}_{h}=\al_h-\sum_{j=1}^n\be^j\be^j_h$ for $h=x,u,\mu$.   Thus,
 \begin{eqnarray}{\label{B2'}}
dB(s)   
&=&\Big(\bar{A}\big(\al_x(\th)Y+\al_u(\th)v\big)
 +\bar{A}\EE'\big(B'\al_{\mu}(\th; \bar{X}')+\bar{A}'\al_{\mu,1}(\th; \bar{X}')Y'\big)+B\al(\th)\Big)ds \nonumber\\
&&+\sum_{j=1}^n\Big(\bar{A}\big(\be^j_x(\th)Y+\be^j_u(\th)v + \EE'(B'\be^j_{\mu}(\th; \bar{X}')+\bar{A}'\be^j_{\mu,1}(\th; \bar{X}')Y')\big)+B\be^j(\th)\Big)dW^j(s)\nonumber\\
&&+\blue{ \int_{U_0}\Big(\bar{A}\big(\ga_x(\bar{X}, \bar{u},z)Y+\ga_u(\bar{X}, \bar{u},z)v\big)+B\gamma(\bar{X}, \bar{u},z)\Big)\tilde{N}(ds,dz)}. 
\end{eqnarray}

To prove (\ref{L2}), we set
\[Z^\ep(s)=Y^\ep(s)-Y(s)\mbox{ and }C^\ep(s)=B^\ep(s)-B(s).\]
Then (\ref{eq0611f}) and (\ref{B1}) imply
\[dZ^\ep(s)=b^\ep(s)ds+\sum_{j=1}^n\si^{\ep,j}(s)dW^j(s)+\int_{U_0}\eta^\ep(s-,z)\widetilde{N}(ds,dz),\]
where
\begin{eqnarray}{\label{bR1}} 
b^\ep(s)&=&b_x(\vartheta^\ep(s))Y^\ep(s)-b_x(\th(s))Y(s)+b_u(\vartheta^\ep(s))v(s)-b_u(\th(s))v(s)\nonumber\\
&&+\EE'\big(B^\ep(s)' b_\mu(\vartheta^\ep(s); X^\ep(s)')-B(s)'b_{\mu}(\th(s); \bar{X}(s)')\big)\nonumber\\
&&+\EE'\big(\bar{A}(s)'b_{\mu,1}(\vartheta^\ep(s); \tilde{X}(s)')Y^\ep(s)'-\bar{A}(s)'b_{\mu,1}(\th(s); \bar{X}(s)')Y(s)'\big).
\end{eqnarray}
We also have similar representations for $\si^{\ep,j}(s)$ and $\eta^\ep(s-,z)$.
We now continue (\ref{bR1}) with 
\begin{eqnarray}{\label{bR2}} 
b^\ep(s)&=&b_x(\vartheta^\ep(s))\big(Y^\ep(s)-Y(s)\big)+\big(b_x(\vartheta^\ep(s))-b_x(\th(s))\big)Y(s)\nonumber\\
&&+\big(b_u(\vartheta^\ep(s))-b_u(\th(s))\big)v(s)\nonumber\\
&&+\EE'\big(\(B^\ep(s)' -B(s)'\)b_\mu(\vartheta^\ep(s); X^\ep(s)')\big)\nonumber\\
&&+\EE'\big(B(s)'\(b_\mu(\vartheta^\ep(s); X^\ep(s)')-b_{\mu}(\th(s); \bar{X}(s)')\)\big)\nonumber\\
&&+\EE'\Big(\bar{A}(s)'b_{\mu,1}(\vartheta^\ep(s); \tilde{X}(s)')\(Y^\ep(s)'-Y(s)'\)\nonumber\\
&&\ \ \ +\bar{A}(s)'(b_{\mu,1}(\vartheta^\ep(s); \tilde{X}(s)')-b_{\mu,1}(\th(s); \bar{X}(s)'))Y(s)'\Big)\nonumber\\
&=&b_x(\vartheta^\ep(s))Z^\ep(s)+\de b_x(s)Y(s)+\de b_u(s)v(s)\nonumber\\
&&+\EE'\big(C^\ep(s)'b_\mu(\vartheta^\ep(s); X^\ep(s)')+\de b_\mu(s;x')B'(s)\big)\nonumber\\
&&+\EE'\(\bar{A}(s)'b_{\mu,1}(\vartheta^\ep(s); \tilde{X}(s)')Z^\ep(s)'+\de b_{\mu,1}(s;X')Y(s)'\bar{A}(s)'\),
\end{eqnarray}
where   $\tilde{X}(s)$ is between 
$X^\ep(s)$ and $\bar{X}(s)$ and 
\begin{eqnarray*}
\de b_x(s)&=&b_x(\vartheta^\ep(s))-b_x(\th(s))\to 0,\\
\de b_u(s)&=&b_u(\vartheta^\ep(s))-b_u(\th(s))\to 0,\\
\de b_\mu(s;\bar{X}')&=&b_\mu(\vartheta^\ep(s); X^\ep(s)')-b_{\mu}(\th(s); \bar{X}(s)'),\\
\de b_{\mu,1}(s;\bar{X}')&=&b_{\mu,1}(\vartheta^\ep(s); \tilde{X}'(s))-b_{\mu,1}(\th(s); \bar{X}(s)').
\end{eqnarray*}
From the boundedness of the $p$th moments of $(\bar{A}(s),\bar{X}(s),Y(s),B(s))$ for  any $p\geq2$, we have
\begin{eqnarray*}
&& \int_0^T\EE |b^\ep(s)|^pds\\
&\leq&K\int_0^T\Big(\EE|b_x(\vartheta^\ep(s))Z^\ep(s)|^p+\EE|\de b_x(s)Y(s)|^p+\EE|\de b_u(s)v(s)|^p\\
 &&\ \ \ +\EE\big(\EE'\big|C^\ep(s)'b_\mu(\vartheta^\ep(s); X^\ep(s)')\big|\big)^p+\EE\big(\EE'\big|\de b_\mu(s;\bar{X}')B'(s)\big|\big)^p\\
&&\ \ \ +\EE\big(\EE'\big|\bar{A}(s)'b_{\mu,1}(\vartheta^\ep(s); \tilde{X}(s)')Z^\ep(s)'\big|\big)^p+\EE\big(\EE'\big|\de b_{\mu,1}(s;\bar{X}')Y(s)'\bar{A}(s)'\big|\big)^p\Big)ds\\
 &\leq& K\int_0^T\Big(\EE|Z^\ep(s)|^p+\EE|\de b_x(s)Y(s)|^p+\EE|\de b_u(s)v(s)|^p\\
 &&\ \ \ +\big(\EE'|C^\ep(s)'|^{q}\big)^{\frac{p}{q}}\big(\EE'|1+X^\ep(s)'|^{\frac{q}{q-1}}\big)^{\frac{p(q-1)}{q}}+\EE\big(\EE'|\de b_\mu(s;\bar{X}')|^p\big)(\EE'|B'(s)|^\frac{p}{p-1})^{p-1}\\
 &&\ \ \ +\EE\big(\EE'|\bar{A}(s)'b_{\mu,1}(\vartheta^\ep(s);\tilde{X}'(s))|^{\frac{p}{p-1}}\big)^{p-1}\EE'|Z^\ep(s)'|^p\\
 &&\ \ \ +\big(\EE'|\bar{A}(s)'Y(s)'|^{\frac{p}{p-1}}\big)^{p-1}\EE\big(\EE'|\de b_{\mu,1}(s;\bar{X}')|^p\big)\Big)ds\\
 &\leq&  K\int_0^T\Big(\EE|Z^\ep(s)|^p+\big(\EE|C^\ep(s)|^{q}\big)^{\frac{p}{q}}+\EE|\de^\ep_1(s)|^p\Big)ds,
\end{eqnarray*}
 where $q=1+\delta$ with $\delta\in (0,1)$   and 
\[\de^\ep_1(s)^p=:|\de b_x(s)Y(s)|^p+|\de b_u(s)v(s)|^p+\EE'|\de b_\mu(s;\bar{X}')|^p+\EE'|\de b_{\mu,1}(s;\bar{X}')|^p.\]
From the expressions for $\de b_x(s)$, $\de b_u(s)$, $\de b_\mu(s;\bar{X}')$ and $\de b_{\mu,1}(s;\bar{X}') $ and the Lipschitz property of $b_\mu$, it is clear that
\[\EE\int^T_0|\de^\ep_1(s)|^pds\to 0,\] 
as $\ep\downarrow 0$. Similar estimate holds for $\int_0^T\EE|\si^{\ep,j}(s)|^pds$. Further, 
\begin{eqnarray*}
&&\EE\int_0^T\int_{U_0}|\eta^\ep(s,z)|^p\nu(dz)ds\\
&\leq&K\EE\int_0^T\int_{U_0}\Big(|\eta_x(\chi^\ep(s),z)Y^\ep(s)-\eta_x(\bar{X}(s), \bar{u}(s),z)Y(s)|^p\\
&&+|\eta_u(\chi^\ep(s),z)-\eta_u(\bar{X}(s), \bar{u}(s),z)|^p|v(s)|^p\Big)\nu(dz)ds\\
&\leq&K\EE\int_0^T\int_{U_0}\Big(|\eta_x(\chi^\ep(s),z)|^p|Y^\ep(s)-Y(s)|^p+|\eta_x(\eta_x(\chi^\ep(s),z)-\eta_x(\bar{X}(s), \bar{u}(s),z)|^p|Y(s)|^p\\
&&+|\eta_u(\chi^\ep(s),z)-\eta_u(\bar{X}(s), \bar{u}(s),z)|^p|v(s)|^p\Big)\nu(dz)ds\\
&=&K\EE\int_0^T|Z^\ep(s)|^p\Big(\int_{U_0}|\eta_x(\chi^\ep(s),z)|^p\nu(dz)\Big)ds\\
&&+K\EE\int_0^T|Y(s)|^p\Big(\int_{U_0}|\eta_x(\chi^\ep(s),z)-\eta_x(\bar{X}(s), \bar{u}(s),z)|^p\nu(dz)\Big)ds\\
&&+K\EE\int_0^T|v(s)|^p\Big(\int_{U_0}|\eta_u(\chi^\ep(s),z)-\eta_u(\bar{X}(s), \bar{u}(s),z)|^p\nu(dz)\Big)ds\\
&\leq&K\int_0^T\big(\EE|Z^\ep(s)|^p+\EE|\delta_2^\ep(s)|^p\big)ds,
\end{eqnarray*}
where
\begin{eqnarray*}
\delta_2^\ep(s)^p&:=&|Y(s)|^p\Big(\int_{U_0}|\eta_x(\chi^\ep(s),z)-\eta_x(\bar{X}(s), \bar{u}(s),z)|^p\nu(dz)\Big)\\
&&+|v(s)|^p\Big(\int_{U_0}|\eta_u(\chi^\ep(s),z)-\eta_u(\bar{X}(s), \bar{u}(s),z)|^p\nu(dz)\Big).
\end{eqnarray*}
We also have
\begin{eqnarray*}
&&\EE\int^T_0|\delta_2^\ep(s)|^pds \\
&\leq& \int_0^T(\EE|Y(s)|^{2p})^{\frac{1}{2}}\(\EE\(\int_{U_0}|\eta_x(\chi^\ep(s),z)-\eta_x(\bar{X}(s), \bar{u}(s),z)|^p\nu(dz)\)^2\)^\frac{1}{2}ds\\
&&+\int_0^T(\EE|v(s)|^{2p})^{\frac{1}{2}}\(\EE\(\int_{U_0}|\eta_u(\chi^\ep(s),z)-\eta_u(\bar{X}(s), \bar{u}(s),z)|^p\nu(dz)\)^2\)^\frac{1}{2}ds\\
&\to&0, 
\end{eqnarray*}
as $\ep\downarrow 0$. Applying   BDG  inequality and Kunita's inequality again, we see that
\begin{eqnarray}\label{eq0611a}
&&\EE\sup_{s\in[0,T]}|Z^\ep(s)|^p\nonumber\\
&\leq&K\Big(\int_0^T\EE\big |b^\ep(s)\big|^pds+\sum_{j=1}^n\int_0^T\EE\big|\si^{\ep,j}(s)\big|^pds\nonumber\\
&&+\EE\int_0^T\int_{U_0}\big|\eta^\ep(s,z)\big|^2(\nu(dz))^{p/2}ds+\EE\int_0^T\int_{U_0}\big|\eta^\ep(s,z)\big|^p\nu(dz)ds
\Big)\nonumber\\
&\leq& K\int^T_0\(\EE|Z^\ep(s)|^p+(\EE|C^\ep(s)|^{q})^{\frac{p}{q}}+\EE(|\de^\ep_1(s)|^p+|\de^\ep_2(s)|^p)\)ds.
\end{eqnarray}

Recalling $\tilde{A}(s)=\ln \bar{A}(s)$, $\tilde{A}^\ep(s)=\ln A^\ep(s)$ and $B(s)=\bar{A}(s)\tilde{B}(s)$, we denote
\[\tilde{C}^\ep(s)=\ep^{-1}(\tilde{A}^\ep(s)-\tilde{A}(s))-\tilde{B}(s).\]
 Similar to above, we have
\[\EE\sup_{s\in[0,T]}|\tilde{C}^\ep(s)|^p\le K\int^T_0\(\EE|Z^\ep(s)|^p+(\EE|C^\ep(s)|^{q})^{\frac{p}{q}}+\EE|\de_3^\ep(s)|^p\)ds,\]
where $|\de_3^\ep(s)|^p$ is similar to $|\de^\ep_1(s)|^p+|\de^\ep_2(s)|^p$ with coefficients $b,\ \si,\ \eta$ replaced by $\tilde{\al},\ \be$.

Since $\tilde{A}^\ep(s)-\tilde{A}(s)=\ep(\tilde{C}^\ep(s)+\tilde{B}(s))$, we have
\begin{eqnarray*}
  C^\ep(s)&=&B^\ep(s)-B(s)\\
  &=&\ep^{-1}(A^\ep(s)-\bar{A}(s))-\bar{A}(s)\tilde{B}(s)\\
  &=&\ep^{-1}\bar{A}(s)\(\frac{A^\ep(s)}{\bar{A}(s)}-1-\ep\tilde{B}(s)\)\\
  &=&\ep^{-1}\bar{A}(s)\(e^{(\ln A^\ep(s)-\ln \bar{A}(s))}-1-\ep\tilde{B}(s)\)\\
  &=&\ep^{-1}\bar{A}(s)\(e^{(\tilde{A}^\ep(s)-\tilde{A}(s))}-1-\ep\big(\tilde{B}(s)+\tilde{C}^\ep(s)\big)\)+\bar{A}(s)\tilde{C}^\ep(s)\\
&\equiv & \Delta^\ep(s)+\bar{A}(s)\tilde{C}^\ep(s),
\end{eqnarray*}
for all $s\in[0,T]$. Note that
\begin{eqnarray*}
\Delta^\ep(s)
&=&\bar{A}(s)\ep^{-1}\(e^{\tilde{A}^\ep(s)-\tilde{A}(s)}-1-(\tilde{A}^\ep(s)-\tilde{A}(s))\)\\
&=&\frac{1}{2}\bar{A}(s)\ep^{-1}e^{\lambda(\tilde{A}^\ep(s)-\tilde{A}(s))}(\tilde{A}^\ep(s)-\tilde{A}(s))^2,
 \end{eqnarray*}
 where $\lambda\in(0,1).$ Since Lemma 4.1 and (\ref{eq0611005A}) imply that
 $\EE|\tilde{A}^\ep(s)-\tilde{A}(s)|^p<K\ep^p$ and 
$$\EE|\bar{A}(s)e^{\lambda(\tilde{A}^\ep(s)-\tilde{A}(s))}|^p=\EE|\bar{A}(s)^{1-\lambda}(A^\ep(s))^\lambda|^p=(\EE|\bar{A}(s)|^{2p(1-\lambda)})^\frac{1}{2}(\EE|A^\ep(s)|^{2p\lambda})^\frac{1}{2}\leq K,$$ 
for any $p\geq 2$,  it is clear that 
\begin{eqnarray*}
\EE\sup_{s\in[0,T]}|\Delta^\ep(s)|^p&\leq& K\ep^{-p}\int_0^T\big(\EE|\bar{A}(s)e^{\lambda(\tilde{A}^\ep(s)-\tilde{A}(s))}|^{2p}\big)^{\frac{1}{2}}\big(\EE|\tilde{A}^\ep(s)-\tilde{A}(s)|^{4p}\big)^{\frac{1}{2}}ds\\
&\leq& K\ep^p\to 0,
 \end{eqnarray*}
 as  $\ep\downarrow 0$.  Thus,
\begin{eqnarray}\label{eq0611b}
\big(\EE\sup_{s\in[0,T]}|C^\ep(s)|^{q}\big)^{\frac{p}{q}} &\le& K\Big(\big (\EE\sup_{s\in[0,T]}|\bar{A}(s)\tilde{C}^\ep(s)|^{q}\big)^{\frac{p}{q}}+\big(\EE\sup_{s\in[0,T]}|\Delta^\ep(s)|^{q}\big)^{\frac{p}{q}}\Big)\nonumber\\
&\le& K\Big(\EE\sup_{s\in[0,T] }|\tilde{C}^\ep(s)|^p+\big(\EE\sup_{s\in[0,T]}|\Delta^\ep(s)|^{q}\big)^{\frac{p}{q}}\Big)\nonumber\\
&\le&K\int^T_0\Big(\EE|Z^\ep(s)|^p+(\EE|C^\ep(s)|^{q})^{\frac{p}{q}}+\sum^3_{i=1}\EE|\de^\ep_i(s)|^p\Big)ds\nonumber\\
&&+K\big(\EE\sup_{s\in[0,T]}|\Delta^\ep(s)|^p).
\end{eqnarray}
The conclusion follows by (\ref{eq0611a}, \ref{eq0611b}).
\end{proof}

{\bf{Proof of the Theorem \ref{smp}}}:
Note that as $\ep\rightarrow0$ 
\begin{eqnarray}\label{eq0611c}
&0\leq&\lim_{\ep\downarrow0}\ep^{-1}\(J(u^\ep)-J(u)\)=\frac{d}{d\ep}J(u+\ep(v-u))|_{\ep=0}\nonumber\\
&\to&\EE\int^T_0\Big(f_x(\ka)Y+f_a(\ka)B +\EE'\big(B'f_\mu(\ka; \bar{X}')+A'f_{\mu,1}(\ka; \bar{X}')Y'\big)+f_u(\ka)v\Big)ds\nonumber\\
&&+\EE\big(\<\Phi_x(\bar{X}(T),\bar{A}(T)),Y(T)\>+\<\Phi_a(\bar{X}(T),\bar{A}(T)),B(T)\>\big).\nonumber\\
\end{eqnarray} 
Denote
\begin{eqnarray}{\label{eq0613adjointa3''}}
g&=&p\al(\th)+\be^{\top}(\th)q+\blue{\int_{U_0}r(z)\ga(\bar{X}, \bar{u},z)\nu(dz)}+f_a(\ka)+\EE'(f_\mu(\ka'; \bar{X}))\nonumber\\
&&+\mathbb{E}'\Big(\big(b_{\mu}^{\top}P'+\sum_{j=1}^n(\si^j_{\mu})^{\top}{Q^j}'+\bar{A}'(p'\al_{\mu}+\be_{\mu}^{\top}q')\big)(\th';\bar{X})\Big),\nonumber\\
 \end{eqnarray}
 and 
 \begin{eqnarray}{\label{eq0613adjointa4''}}
G&=&b^{\top}_x(\th)P+\sum_{j=1}^n\(\si^j_x(\th)\)^{\top} Q^j+\int_{U_0}\eta_x^{\top}(\bar{X},\bar{u}, z)R(z)\nu(dz)+f_x(\ka)\nonumber\\
&&+\bar{A}\(p\al_x^{\top}(\th)+q\be_x^{\top}(\th))+\blue{\int_{U_0}r(z)\ga_x(\bar{X}, \bar{u},z)\nu(dz)}\)+\EE'\big(\bar{A}f_{\mu,1}^{\top}(\ka';\bar{X})\big)\nonumber\\
&&+\mathbb{E}' \Big(\bar{A}\big(b_{\mu,1}^{\top}P'+\sum_{j=1}^n(\si^j_{\mu,1})^{\top}{Q^j}'+\bar{A}'\big(\al_{\mu,1}^{\top}p'+\be_{\mu,1}^{\top}q'\big)\big)(\th';\bar{X})\Big).
 \end{eqnarray}
  Then,
  \begin{equation*}
\left \{
  \begin{aligned}
   dp(s)=&-g(s)ds+q(s) dW(s)+\int_{U_0}r(s-,z)\widetilde{N}(ds,dz),&\\
   p(T)=&\Phi_a(\bar{X}(T),\bar{A}(T)),&
      \end{aligned}
  \right.
\end{equation*}
and 
\begin{equation*}
\left \{
  \begin{aligned}
   dP(s)=&-G(s)ds+Q(s) dW(s)+\int_{U_0}R(s-,z)\widetilde{N}(ds,dz),&\\
   P(T)=&\Phi_x(\bar{X}(T),\bar{A}(T)).&
      \end{aligned}
  \right.
\end{equation*}
Applying It\^o's formula to  $\<p(s),B(s)\>$ and $\<P(s),Y(s)\>$, we have 
\begin{eqnarray*}
&&d\<p(s),\ B(s)\>\\
&=&p(s-)dB(s)+B(s-)dp(s)+d[p,B](s)\\   
&=&p\Big(\bar{A}\big(\al_x(\th)Y+\al_u(\th)v\big)+B\al(\th) +\bar{A}\EE'\big(B'\al_{\mu}(\th; \bar{X}')+\bar{A}'\al_{\mu,1}(\th; \bar{X}')Y'\big)\Big)ds\\
&&+\sum_{j=1}^np\Big(\bar{A}\big(\be^j_x(\th)Y+\be^j_u(\th)v\big)+B\be^j(\th)+\bar{A}\EE'\big(B'\be^j_{\mu}(\th; \bar{X}')+\bar{A}'\be^j_{\mu,1}(\th; \bar{X}')Y'\big)\Big)dW^j(s)\\
&&\blue{+\int_{U_0}\Big(p\bar{A}\big(\ga_x(\bar{X}, \bar{u},z)Y+\ga_u(\bar{X}, \bar{u},z)v\big)+pB\ga(\bar{X}, \bar{u},z)\Big)\tilde{N}(ds,dz)}\\
&&-Bgds+Bq^{{\top}}dW(s)+B(s-)\int_{U_0}r(s-,z)\widetilde{N}(ds,dz)\\
&&+\sum_{j=1}^nq^j\Big(\bar{A}\big(\be_x^j(\theta)Y+\be_u^j(\theta)v\big)+B\be^j(\theta)+\bar{A}\EE'\big(B'\be^j_\mu(\theta;\bar{X}')+\bar{A}'\be_{\mu,1}^j(\theta;\bar{X}')Y'\big) \Big)ds\\
&&\blue{+\int_{U_0}r(s-,z)\Big(\bar{A}\big(\gamma_x(\bar{X}, \bar{u},z)Y+\gamma_u(\bar{X}, \bar{u},z)v \big)+ B\ga(\bar{X}, \bar{u},z)\Big)N(ds,dz)},
\end{eqnarray*}
and 
\begin{eqnarray*}
&&d\<P(s),\ Y(s)\>\\
&=&P(s-)dY(s)+Y(s-)dP(s)+d[P,Y](s)\\   
&=&\Big<P,\  b_x(\th)Y+b_u(\th)v
 +\EE'\big(B'b_{\mu}(\th; \bar{X}')+\bar{A}'b_{\mu,1}(\th; \bar{X}')Y'\big)\Big>ds\\
&&+\Big<P, \ \sum_{j=1}^n\big(\si_x^j(\th)Y+\si_u^j(\th)v
 +\EE'\big(B'\si_{\mu}^j(\th; \bar{X}')+\bar{A}'\si_{\mu,1}^j(\th; \bar{X}')Y'\big)\big)dW^j(s)\Big>\\
&&+\int_{U_0}\Big<P(s-),\  \big(\eta_x(\bar{X}, \bar{u},z)Y+\eta_u(\bar{X}, \bar{u},z)v
\big)\Big>\tilde{N}(ds,dz)\\
&&+\Big<Y(s-),\ -G(s)ds+Q^j(s)dW^j(s)+\int_{U_0}R(s-,z)\widetilde{N}(ds,dz)\Big>\\
&&+\sum_{j=1}^n\Big<Q^j, \ \si_x^j(\th)Y+\si_u^j(\th)v +\EE'\big(B'\si_{\mu}^j(\th; \bar{X}')+\bar{A}'\si_{\mu,1}^j(\th; \bar{X}')Y'\big)\Big>ds\\
&&+\int_{U_0}\Big<R(s-,z),\ \eta_x(\bar{X},\bar{u},z)Y+\eta_u(\bar{X}, \bar{u},z)v \Big>N(dz,ds).
\end{eqnarray*}
Hence,
\begin{eqnarray}\label{PB4}
&&\EE\big(\<p(T),\ B(T)\>+\<P(T),\ Y(T)\>\big)\nonumber\\
&=&\int_0^T\EE\big(d\<p(s),\ B(s)\>+d\<P(s),Y(s)\>\big)\nonumber\\
&=&\int_0^T\EE\Big(\Big<\bar{A}p\al_x^{{\top}}(\th)-G+ \bar{A}\EE'\big(\bar{A}'\al_{\mu_1}^{{\top}}(\th';\bar{X})p'+b_{\mu,1}^{{\top}}(\th';\bar{X})P'\big),\ Y\Big>\Big)ds\nonumber\\
&&+\int_0^T\EE\Big(\Big<\int_{U_0}\big(\blue{r(z)\bar{A}\gamma_x^\top(\bar{X}, \bar{u},z)}+\eta_x^{{\top}}(\bar{X}, \bar{u},z)R(z)\big)\nu(dz),\ Y\Big>\Big)ds\nonumber\\
&&+\int_0^T\EE\Big(\Big<\bar{A}\EE'\Big(\sum_{j=1}^n\bar{A}'(\be_{\mu,1}^j(\th';\bar{X}))^{{\top}}(q^j)'+\sum_{j=1}^n(\si_{\mu,1}^j(\th';\bar{X}))^{{\top}}(Q^j)'\Big),\ Y
\Big>\Big)ds\nonumber\\
&& +\int_{0}^T\EE\Big(\Big<b_x^{{\top}}(\th)P+\sum_{j=1}^n\(\si_x^j(\th))^{{\top}}Q^j+\bar{A}(\be_x^j(\th))^{{\top}}q^j\)
, \ Y\Big>\Big)ds\nonumber\\
&&+\int_0^T\EE\Big(\Big<p\al(\th)\blue{+\int_{U_0}r(z)\ga(\bar{X}, \bar{u},z)\nu(dz)}-g
, \ B\Big>\Big)ds\nonumber\\
&&+\int_0^T\EE\Big(\EE'\big(\bar{A}'p'\al_{\mu}(\th';\bar{X})+b_{\mu}^{{\top}}(\th';\bar{X})P'\big)
, \ B\Big>\Big)ds\nonumber\\
&&+\int_0^T\EE\Big(\Big<\sum_{j=1}^n\Big(q^j\be^j(\th)+\EE'\big((\si_{\mu}^j(\th';\bar{X}))^{{\top}}Q'^j +\bar{A}'q'^j\be^j_{\mu}(\th';\bar{X})\big)\Big), \ B\Big>\Big)ds\nonumber\\
&&+\int_0^T\EE\Big(\Big<\bar{A}p\al_u^\top(\th)+b_u^\top (\th) P+\sum_{j=1}^n\big(\bar{A}q^j(\be^j_u(\th))^\top+(\si_u^j(\th))^\top Q^j\big)\nonumber\\
&&\ \ \ \ +\int_{U_0}\big(\blue{r(z)\bar{A}\ga_{u}^\top(\bar{X}, \bar{u},z)}+\eta^\top _u(\bar{X}, \bar{u},z)R(z)\big)\nu(dz), \ v\Big>\Big)ds{\blue{.}}
\end{eqnarray}

Combining (\ref{eq0611c}), (\ref {eq0613adjointa3''}) (\ref{eq0613adjointa4''}) and (\ref{PB4}),  we can get 
\begin{eqnarray*}
0&\le&\EE\int^T_0\Big<\bar{A}\Big(p\al_u^\top(\th)+\be_u^\top (\th)q\blue{+\int_{U_0}r(z)\gamma_u^\top(\bar{X}, \bar{u},z)\nu(dz)}\Big)+b_u^\top (\th)P\\
&&\ \ \ \ \ \ \ \ +\sum_{j=1}^n(\si^j_u(\th))^\top Q^j+f_u^\top(\ka)+\int_{U_0}\eta_u^\top (\bar{X},\bar{u},z)R(z) \nu(dz),\  v\Big>ds.
\end{eqnarray*}
By the definition of Hamiltonian function $\cH$, we arrive at 
\begin{equation}\label{ZQH}
0\le\EE\int^T_0\big<\cH_{\blue{u}}\big(\bar{X},\bar{A},\mu^{\bar{X},\bar{A}},\bar{u},p,q,\blue{ r(\cdot)},P,Q,R(\cdot)\big), \ v \big>ds.
\end{equation}
\blue{The arbitrariness of $v$} then implies
\[\cH_{\blue{u}}(\bar{X}(s),\bar{A}(s),\mu^{\bar{X},\bar{A}}(s),\bar{u}(s),p(s),q(s), \blue{r(s,z)},P(s),Q(s),R(s,z))=0{\blue{,}}\]
for a.e. $s\in[0,T], \ z\in U_0$, $\PP$- a.s. This complete the proof of the Theorem \ref{smp}.
\qed

\section{Sufficient condition} \label{sec5}
\setcounter{equation}{0}
\renewcommand{\theequation}{\thesection.\arabic{equation}}

{\bf{Proof of the Theorem \ref{Sufficient}}}:
For any $v(\cdot)\in\cU$ with the corresponding controlled state process $(X^v(\cdot),A^v(\cdot)), $ the convexity of $\Phi$ implies
\begin{eqnarray*}
&&J(v)-J(\bar{u})\\
&\geq&\EE\Big(\big<\Phi_x(\bar{X}(T),\bar{A}(T)),\ X^v(T)-\bar{X}(T)\big>+\big<\Phi_a(\bar{X}(T),\bar{A}(T)), \ A^v(T)-\bar{A}(T)\big>\Big)\\
&&+\EE\left(\int_0^T \big(f(\kappa^v)-f(\kappa)\big)ds\right)\\
&=&\EE\Big(\big<P(T), \ X^v(T)-\bar{X}(T)\big>+\big<p(T), \ A^v(T)-\bar{A}(T)\big>\Big)\\
&&+\EE\left(\int_0^T \big(f(\kappa^v)-f(\kappa)\big)ds\right).
\end{eqnarray*}
By the definition of the Hamiltonian $\cH$ of (\ref{eq0611H}) and Equation  (\ref{hmdhx1}), we have
\begin{eqnarray}\label{SSMP2}
&&\EE\int_{0}^T\big(f(\kappa^v)-f(\kappa)\big)ds\nonumber\\
&=&\EE\int_{0}^T\big(\cH(k^v,p,q,\blue{r(\cdot)},P,Q,R(\cdot))-\cH(s)\big)ds\nonumber\\
&&-\EE\int_0^T\(\<p,  A^v\alpha(\theta^v)-\bar{A}\alpha(\theta)\>+\<q, A^v\beta(\theta^v(s))-\bar{A}\beta(\theta)\>\)ds\nonumber\\
&&-\blue{\EE\int_0^T\int_{U_0}\big<r(z), A^v\ga(X^v, v,z)-\bar{A}\ga(\bar{X},\bar{u},z)\big>\nu (dz)ds}\nonumber\\
&&-\EE\int_0^T\Big(\big<P,\  b(\theta^v)-b(\theta)\big>+ \tr[Q^\top \si(\th^v)]-\tr[Q^\top \si(\th)]\Big)ds\nonumber\\
&&-\EE\int_0^T\int_{U_0}\big<R(z),\  \eta(X^v, v,z)-\eta(\bar{X}, \bar{u},z )\big>\nu (dz)ds.
\end{eqnarray}
Recall that $(p(\cdot), q(\cdot), r(\cdot, \cdot))$ and $(P(\cdot), Q(\cdot), R(\cdot, \cdot))$ satisfy equations (\ref{eq0613adjointa3}) and (\ref{eq0613adjointa4}), respectively, applying It\^o's formula to $\<P(s),\ X^v(s)-\bar{X}(s)\>$ and $\<p(s),\ A^v(s)-\bar{A}(s)\>$ gives
 \begin{eqnarray}\label{SSMP3}
&&\EE\Big(\big<P(T), \ X^v(T)-\bar{X}(T)\big>+\big<p(T),\  A^v(T)-\bar{A}(T)\big>\Big)\nonumber\\
&=&\EE\int_0^T-\big<\cH_x(s)+\bar{A}\mathbb{E}'(\cH_{\mu,1}(s';\bar{X})),\ X^v-\bar{X}\big>ds\nonumber\\
&&-\EE\int_0^T\big<\cH_a(s)+\mathbb{E}'(\cH_\mu(s';\bar{X})),\ A^v-\bar{A}\big>ds\nonumber\\
&&+\EE\int_0^T\Big(\big<P,\  b(\theta^v)-b(\theta)\big>+\sum_{j=1}^n\big((\si^j(\th^v))^\top-(\si^j(\th))^\top \big)Q^j\Big)ds\nonumber\\
&&+\EE\int_0^T\int_{U_0}\big<R(z),\  \eta(X^v, v, z)-\eta(\bar{X}, \bar{u},z)\big>\nu (dz)ds\\
&&+\EE\int_0^T\(\<p,  A^v\alpha(\theta^v)-\bar{A}\alpha(\theta)\>+\<q, A^v\beta(\theta^v)-\bar{A}\beta(\theta)\>\)ds\nonumber\\
&&\blue{+\EE\int_0^T\int_{U_0}\big<r(z),\  A^v(s)\ga(X^v, v,z)-\bar{A}\ga(\bar{X}, \bar{u},z)\big>\nu (dz)ds}\nonumber.
\end{eqnarray}
Equations (\ref{SSMP2}) and (\ref{SSMP3}) together imply  
 \begin{eqnarray}\label{SSMP6}
J(v)-J(\bar{u})
&\geq&\EE\int_{0}^T\Big(\cH(\ka^v,p,q, r(\cdot),P,Q,R(\cdot))-\cH(s)\Big)ds\nonumber\\
&&-\EE\int_0^T\Big(\big<\cH_x(s),\ X^v(s)-\bar{X}(s)\big>+\big<\cH_a(s),\ A^v-\bar{A}\big>\Big)ds\\
&&-\EE\int_0^T\mathbb{E}' \Big(\big<\cH_\mu(s;\bar{X}'),\ {A^v}'-\bar{A}'\big>+\big<\bar{A}'\cH_{\mu,1}(s;\bar{X}'),\ {X^v}'-\bar{X}'\big>\Big)ds,\nonumber
\end{eqnarray}
where $\cH_{\mu}(s;\bar{X}')=\cH_{\mu}(\ka,p,q, r(\cdot),P,Q,R(\cdot);\bar{X}')$. Since $\cH$ is convex with respect to $(x,a,\mu,u)$, we get 
\begin{eqnarray}\label{SSMP4}
&&\cH\big(\ka^v(s),p(s),q(s), r(s,\cdot),P(s),Q(s),R(s,\cdot)\big)-\cH(s)\nonumber\\
&\geq&\big<\cH_x(s),\ X^v(s)-\bar{X}(s)\big>+\big<\cH_a(s),\ A^v(s)-\bar{A}(s)\big>+\big<\cH_u(s), \ v(s)-\bar{u}(s)\big>\nonumber\\
&&+\mathbb{E}'\Big(\big<\cH_{\mu}(s;\bar{X}'),\ A^v(s)'-\bar{A}(s)'\big>+\big<A'(s)\cH_{\mu,1}(s;\bar{X}'),\ {X^{v}(s)}'-\bar{X}(s)'\big>\Big).
\end{eqnarray}

 Taking expectation on both sides of (\ref{SSMP4}) and plugging into (\ref{SSMP6}), we have  
\begin{equation*}
J(v)-J(\bar u)\geq\EE\int_0^T\<\cH_u(s),\ v(s)-\bar{u}(s)\>ds=0,
\end{equation*}
which leads to the desired result.
\qed

\section{Three examples}\label{sec6}
\setcounter{equation}{0}
\renewcommand{\theequation}{\thesection.\arabic{equation}}

In this section, we present three examples, one for the motivation of this article, another for potential application, and the last for a solvable case of the
SMP obtained.

\begin{exmp}[Mean-variance portfolio selection problem]
We consider a market that consists of one risk-free asset and one risky asset.  The price of the risk-free asset $S_0(t)$ evolves according to the deterministic differential equation:
\[dS_0(t)=rS_0(t)dt, \ \ \ S_0(0)=s_0>0,\]
where $r>0$ is the risk-free interest rate. The asset prices do not evolve in isolation. Instead, they are often affected by the overall behavior of investors, such as total investment willingness and average wealth level. Therefore,  the other asset prices are often expressed as:
\[dS(t)=S(t)\(b(t,\mu(t))dt+\si(t, \mu(t))dW(t)+\int_{U_0}\eta(t,z)\tilde{N}(dt,dz)\),\ \ \ S(0)=s,\]
where $s>0$ is the initial price,   $b(\cdot)$ is the rate of return of the asset, $\si(\cdot)$ and $\eta(\cdot)$ are
 the volatilities of the asset and  $\mu(\cdot)$ is a term representing the average behavior of the investors, $\tilde{N}$ is a compensated 
 Poisson random measure on a measurable space $U_0$ describing the sudden change of the price due to some extreme events.  Then, the  wealth process of the $i$-th investor is given by:
\begin{eqnarray*}
dX_i(t)&=&X_i(t)[u_i(t)(b(t, \mu(t))-r)+r]dt+ X_i(t)u_i(t)\si(t, \mu(t))dW(t)\\
&&+\int_{U_0}X_i(t)u_i(t)\eta(t,z)\tilde{N}(dt dz),\qquad
i=1,2,\cdot\cdot\cdot\end{eqnarray*}
 where $u_i(\cdot)$ is the risky asset investment ratio of the $i$-th investor.  Traditionally, $\mu(\cdot)$  is modeled as the average of investor states:$$\mu(\cdot)=\lim_{N\to\infty}\frac{1}{N}\sum_{i=1}^N X_i(t),$$ which can be more generally expressed using the empirical distribution: 
\[\mu(\cdot)= \lim_{N\to\infty}\frac{1}{N}\sum_{i=1}^N \delta_{X_i(\cdot)},\]
where $\delta_{X_i(\cdot)}$  is the Dirac measure  at $X_i(\cdot)$, representing the distribution of the wealth across all investors.

However, in real markets, not all investors have the same influence on the market. Large-scale investors such as institutional fund managers exert significantly more influence on market trends compared to small-scale investors. Thus, market consensus or sentiment should be shaped by a wealth-weighted average of investor opinions, not a simple average. Accordingly, we consider the weighted empirical distribution, i.e. $$\mu(\cdot)=\lim_{N\to\infty}\frac{1}{N}\sum_{i=1}^N A_i(\cdot) \delta_{X_i(\cdot)}(t),$$
where $A_i(\cdot)>0 $  represents the weight (e.g., capital scale or influence) of the $i$-th investor.
 This reflects a more realistic depiction of financial markets: the influence of an investor should be proportional to the amount of capital they manage.

We drop the subscript $i$ and consider a typical investor with state equation
\begin{eqnarray*}
dX(t)&=&X(t)[u(t)(b(t, \mu(t))-r)+r]dt+ X(t)u(t)\si(t, \mu(t))dW(t)\\
&&+\int_{U_0}X(t)u(t-)\eta(t,z)\tilde{N}(dt dz),
\end{eqnarray*}
the weight
\[dA(t)=A(t)\(\al(t)dt+\be(t)dW(t)+\int_{U_0}\ga(t-,z)\tilde{N}(dt dz)\),\]
and the cost functional
\[J(u)=\EE\((X(T)-\EE(X(T))^2-\la \EE(X(T))\),\]
where $\la>0$ is a suitable constant, and
\[\mu(t)=\EE\(A(t)\de_{X(t)}\).\]
\end{exmp}

It is clear that the problem in the example above is a special case of the weighted mean-field control problem
we studied in previous sections and can be solved using the SMP we derived.

The next example demonstrate that the problem which is not weighted to begin with and need to be
converted to a weighted mean-field problem for its solution. For simplicity of notation, we do not consider the jump case.

\begin{exmp}[Mean-field control problem with partial information]
Consider a control problem with mean-field SDE as its state equation:
\[dX(t)=b\big(t,X(t),u(t),\EE (X(t))\big)dt+\si\big(t,X(t),u(t),\EE( X(t))\big)dW(t),\]
and cost functional
\[J(u)=\EE\(\int^T_0 f(t,X(t),u(t),\EE X(t))dt+\Phi(X(T))\).\]
Suppose that $X=X^u$ is not observable. Instead, we observe a process $Y^u$ governed by 
\[Y^u(t)=\int^t_0 h(s,X(s))ds+B(t),\]
where $(B,W)$ is a 2-dimensional Brownian motion. The control $u$ must be $\cF^Y_t$-adapted.

To break the circular dependence between the control $u$ and the observation $Y^u$, a change of measure technique is usually 
taken when $h$ is bounded. Namely, we define a probability measure $Q$ such that
\[\frac{dQ}{dP}=\exp\(\int^T_0 h(s,X(s))dB(s)-\frac12\int^T_0|h(s,X(s))|^2ds\)\equiv M(T).\]
Under $Q$, $(Y,W)$ is a Brownian motion. Thus, the original problem is converted to a weighted mean-field control problem with partial 
information with state
\[\left\{\begin{array}{ccl}
dX(t)&=&\tilde{b}(t,X(t),u(t),\mu(t))dt+\tilde{\si}(t,X(t),u(t),\mu(t))dW(t),\\
dM(t)&=&-M(t)h(t,X(t))dY(t),
\end{array}\right.\]
and cost functional
\[J(u)=\EE^Q\(\int^T_0f(t,X(t),u(t))M(t)dt+\Phi(X(T)M(T)\),\]
where 
\[\tilde{b}(t,x,u,\mu)]=b\(t,x,u,\frac{\int x\mu(dx)}{\mu(\RR)}\),\]
and $\tilde{\si}$ is defined similarly.
\end{exmp}

The problem above is a {\em conditional} weighted mean-field control problem. It can be solved by a {\em conditional} SMP which can
 be derived similar to the SMP we obtained in this article.

Finally, to demonstrate the application of the stochastic maximum principle, we consider the following solvable example.  To avoid overcomplicating the already notationally heavy presentation, we assume $d=n=k=1$.

\begin{exmp}[Linear quadratic case]\label{ex0509a}
 Consider the following dynamic system
\begin{equation}\label{exam}
\left\{\begin{array}{ccl}
dX(t)&=&\big(b_{11}X+b_{12}\<\mu, \iota\>+b_{13}u\big)dt +\big(\si_{11}X+\si_{12}\<\mu,  \iota\>+\si_{13}u\big)dW(t),\\
dA(t)&=&A(t)\big(\al dt+\be dW(t)\big),\\
X(0)&=&x,\ A(0)=a,\ t\in[0,T],
\end{array}\right.
\end{equation}
where $b_{1i}, \si_{1i}, i=1,2,3$ are bounded deterministic functions defined on $[0, T]$. Here, $x, a, \al$ and $\be $ are given constants and the function $ \iota(x)=x$.  Suppose that $b_{13}+\be\si_{13}=0$.

Furthermore, we consider the following cost functional
\begin{eqnarray}
  J(u) =\EE\Big(\int_0^T\big(R_1X^2(t)+R_2u^2(t)\big)dt+\Phi X(T)^2\Big), 
\end{eqnarray}
where $R_1\geq0, R_2>0$ and $\Phi$ are deterministic constants.
\end{exmp}

Note that $ \iota^*(\mu)\equiv\<\mu, \iota\>=\EE[AX]$. From Definition \ref{def2.2} of the weighted measure derivative, it follows that
$$ \iota^*(\mu+\ep\nu)- \iota^*(\mu)=\<\mu+\ep\nu, \iota\>-\<\mu,  \iota\>=\ep\<\nu, \iota\>.$$
Hence, $ \iota^*_\mu(\mu;x)=x$. In this case,  the Hamiltonian functional is given by
\begin{eqnarray*} 
&&\cH\big(x,a,\mu, u, p,q, P,Q)\nonumber\\
&\equiv&a\big(p\al+q\be\big)+P\big(b_{11}x+b_{12}\<\mu, \iota\>+b_{13}u\big)\nonumber\\
&&\ \ +Q\big(\si_{11}x+\si_{12}\<\mu, \iota\>+\si_{13}u\big)+R_1 x^2+R_2u^2,
\end{eqnarray*}
where $(p,q)$ and  $(P,Q)$ are solutions to the following equations
\begin{equation}\label{Exad1}
\left \{
\begin{split}
  \begin{aligned}
   dp(t)=&-\big\{p\al+q\be+\bar{X}\mathbb{E}[Pb_{12}+Q\si_{12}]\big\}dt+q(t) dW(t),&\\
   p(T)=&0,&
      \end{aligned}
  \end{split}
  \right.
\end{equation}
and 
\begin{equation}\label{Exad2}
\left \{
\begin{split}
  \begin{aligned}
   dP(t)=&-\big\{Pb_{11}+Q\si_{11}+2R_1 \bar{X}+A\EE[Pb_{12}+Q\si_{12}]\big\}dt+Q(t) dW(t),&\\
   P(T)=&2\Phi \bar{X}(T).&
      \end{aligned}
  \end{split}
  \right.
\end{equation}
By  Theorem \ref{smp},  our optimal control $\bar{u}$ satisfies the following stationarity conditions
\[Pb_{13}+Q\si_{13}+R_2 \bar{u}=0.\]
In other word,
\begin{equation}\label{eq0509a}\bar{u}(t)=-R_2^{-1}(Pb_{13}+Q\si_{13}),\end{equation}
where $(P, Q)$ is  the unique solution to the equation (\ref{Exad2}) and  $\bar{X}$ is the optimal state trajectory.

Applying It\^o's formula to $A(t)\bar{X}(t)$, we have 
\begin{equation*}
\left\{\begin{array}{ccl}
dA(t)\bar{X}(t)&=&\Big((\al+b_{11}+\si_{11}\be)A\bar{X}+(b_{12}+\si_{12}\be)A\EE[A\bar{X}]\Big)dt\\
 &&+\big( (\be+\si_{11})A\bar{X}+\si_{12}A\EE[A\bar{X}]+\si_{13}A\bar{u}\big)dW(t),\\
\bar{X}(0)A(0)&=&xa,\ t\in[0,T].
\end{array}\right.
\end{equation*}
Hence,
 \begin{equation*}
\left\{\begin{array}{ccl}
d\EE[A\bar{X}(t)]&=&\Big((\al+b_{11}+\si_{11}\be)\EE[A\bar{X}]+(b_{12}+\si_{12}\be)\EE[A]\EE[A\bar{X}]\Big)dt,\\
\EE[\bar{X}(0)A(0)]&=&xa,\ t\in[0,T].
\end{array}\right.
\end{equation*}
Note that $\EE(A)$ is clearly computable. Thus, we can solve $\EE(A\bar{X})$ explicitly. As a consequence, $(\bar{X},\ A,\ \bar{u},\ p,\ q,\ P,\ Q)$ can
 be solved explicitly from the linear FBSDE (\ref{exam}, \ref{Exad1}, \ref{Exad2}, \ref{eq0509a}).
 
{\color{blue} \begin{remark}
 A solvable example can also be provided for jump case with $W(\cdot)$ in the last example being replaced by a compensated
 Poisson process. We omit that due to the length of the paper.
 \end{remark}}

\section{Concluding remarks}\label{sec7}
\setcounter{equation}{0}
\renewcommand{\theequation}{\thesection.\arabic{equation}}

In this article, we have studied the optimal control problem
for weighted mean-field system with jumps. To overcome the difficulty caused by the non-Lipschitz continuity of the coefficients, 
we introduced an $L_{p,q}$-estimation technique for the proof of the existence and uniqueness of the solution to the system, as well as the expansion 
estimate associated with the convex perturbation of the optimal control. We also point out that in our definition of the directional derivative $f_\mu(\mu;x)$
of a functional of the measure, we do not need it to be bounded in $x$. Instead, we assumed it to be at most linear growth in $x$,
 and defined the Wasserstein distance accordingly (used the space $\LL_1(\RR^d)$). It is possible to relax this linear growth 
 condition to polynomial growth without much difficulty. We decide not to do so for the simplicity of notation. This is one of the main difference from \cite{T-X-2023}. This extension  substantially broaded the scope of the applications. For example, the easy and interesting linear case  considered in Example \ref{ex0509a} is not coverred by the results of \cite{T-X-2023} because $\iota^*_\mu(\mu,x)=x$ is not bounded.
 
 Another main difference from \cite{T-X-2023} is the inclusion of the jump terms. As we have indicated in the introduction,
  the study of the optimal control of the jump diffusion has attracted substantial research attention in the literature. The problems
   for jump-diffusions involving weighted mean-field interaction are certainly worthy of study. In fact, more delicated 
   estimates are needed in the jump case.
 
 Finally, we leave 
 the case with unbounded coefficients ($\al,\ \be,\ \ga$) for the weight equation as a challenging {\em open} problem.
  We hope to come back to this problem in our future research.

 \end{document}